\documentclass[11pt]{article}

\usepackage[a4paper,margin=1.15in]{geometry}
\usepackage{amsmath,amssymb,amsthm,mathtools}
\usepackage{enumitem}
\usepackage{microtype}
\usepackage{cite}
\usepackage[hidelinks]{hyperref}
\hypersetup{pdftitle={Pullback metrics and homological obstructions to BLD remetrization of branched covers},pdfauthor={Deguang Zhong}}

\numberwithin{equation}{section}
\newtheorem{theorem}{Theorem}[section]
\newtheorem{lemma}[theorem]{Lemma}
\newtheorem{corollary}[theorem]{Corollary}
\newtheorem{proposition}[theorem]{Proposition}
\newtheorem{question}[theorem]{Question}
\theoremstyle{remark}
\newtheorem{remark}[theorem]{Remark}
\newtheorem{definition}[theorem]{Definition}
\newtheorem{construction}[theorem]{Construction}
\newtheorem{example}[theorem]{Example}

\newcommand{\R}{\mathbb R}
\newcommand{\C}{\mathbb C}
\newcommand{\Z}{\mathbb Z}
\newcommand{\Sn}{\mathbb S^n}
\newcommand{\HH}{\mathcal H}
\DeclareMathOperator{\diam}{diam}
\DeclareMathOperator{\dist}{dist}
\DeclareMathOperator{\id}{id}
\DeclareMathOperator{\sep}{sep}
\DeclareMathOperator{\vol}{vol}

\setlist[enumerate,1]{label=\textup{(\roman*)},leftmargin=2.1em,itemsep=3pt}
\title{\bfseries Pullback metrics and homological obstructions to BLD remetrization of branched covers}

\author{%
Deguang Zhong\\[0.4em]
\small Institute of Applied Mathematics, Shenzhen Polytechnic University,
Shenzhen 518055, China\\
\small \texttt{huachengzhon@163.com}
}

\date{}

\begin{document}

\maketitle

\begin{abstract}
We characterize admissible BLD remetrizations of branched self-covers of spheres by linear local contractibility of their canonical pullback length metrics. A relative homology obstruction involving the total degree of an inverse image component yields uniform porosity of branch values. Explicit degree-two cusp covers of $\mathbb S^n$, $n\ge3$, have uniformly porous branch values of Hausdorff dimension $n-2$ but admit no compatible Ahlfors $n$-regular, linearly locally contractible BLD source metric. Short homologically essential loops provide the obstruction, showing that porosity is necessary but insufficient for admissibility. We also prove quasisymmetric invariance and sharp degree dependence of porosity for sphere-to-sphere covers in dimension two.
\end{abstract}



\section{Introduction}\label{sec1}

\subsection{The source metric problem}
Question~27 in the problem list of Heinonen and Semmes~\cite{HS1997} asks whether a prescribed branched cover can be made geometrically regular by changing the metric in its domain. We use the fixed target formulation, also stated in~\cite{GW2017}. Throughout, $\rho$ is the geodesic distance on the unit round sphere; we write $\rho_n$ when the dimension needs to be indicated.

\begin{question}\label{que1.1}
Let $n\geq3$ and let $f\colon\Sn\to\Sn$ be an orientation-preserving branched cover. Does there exist a metric $d$ inducing the original topology on the source sphere such that $(\Sn,d)$ is Ahlfors $n$-regular and linearly locally contractible, and $f\colon(\Sn,d)\to(\Sn,\rho)$ has bounded length distortion?
\end{question}

A metric as in Question~\ref{que1.1} is called \emph{admissible} for $f$. The target metric is fixed, and a source metric is \emph{compatible} if it induces the original topology. Our BLD convention requires the length inequalities for every continuous path, with infinite lengths allowed; see~\eqref{eq2.3}.

A negative answer was announced by Guo and Williams in~\cite[Corollary~3.3]{GW2016a} and stated in~\cite[Corollary~1.6]{GW2017}; their pullback paper~\cite[Theorem~8.6]{GW2016b} also treats the obstruction from positive measure branch values. We make no claim to the first negative answer. The purpose here is to connect an explicit essential loop obstruction with the canonical metric criterion and the quantitative restrictions on branch values. This connection separates two issues: whether branch values have holes at every small scale, and whether the inverse images of those neighborhoods can be contracted at a comparable scale.

\subsection{An explicit homological obstruction}
The first result constructs the cover directly and excludes a larger class of source metrics than those in Question~\ref{que1.1}.

\begin{theorem}\label{thm1.2}
For every $n\geq3$ there is an orientation-preserving degree-two branched cover
\[
 F_n\colon\Sn\longrightarrow\Sn
\]
for which no compatible source metric is simultaneously quasiconvex, linearly locally contractible, and makes $F_n$ BLD to the fixed round target.
\end{theorem}

\begin{corollary}\label{cor1.3}
For every $n\geq3$ there is a branched cover $\Sn\to\Sn$ admitting no compatible Ahlfors $n$-regular, linearly locally contractible source metric for which it is BLD. Thus the answer to Heinonen--Semmes Question~27 is negative in every stated dimension.
\end{corollary}

In dimension three we postcompose the standard branched double cover with an explicit ambient homeomorphism whose image of the branch value circle has the two arcs $(s,\pm s^2,0)$ near a cusp. For each small $t$, a segment joining these arcs has length comparable to $t^2$, and its full inverse image is an essential circle in the inverse image of a target neighborhood of radius comparable to $t$. The latter inverse image is an annulus times an interval, as verified by an explicit algebraic model in Lemma~\ref{lem3.2}. The lower BLD inequality would make the circle short in the unknown source metric, whereas quasiconvexity and the upper inequality would keep any controlled contraction inside that inverse image. This contradicts its nonzero first homology class. Iterated suspension preserves the obstruction on spheres in every dimension $n\ge3$.

The proof of Theorem~\ref{thm1.2} in Section~\ref{sec3} uses neither a pullback metric nor porosity nor an approximate homotopy inverse. The only additional input for Corollary~\ref{cor1.3} is that an Ahlfors $n$-regular, linearly locally contractible metric sphere is quasiconvex (Corollary~\ref{cor2.3}). Independence here refers to this chain of implications; it is not a priority claim for the construction or for the negative answer.

\subsection{The canonical metric criterion}
Let $M$ be a closed connected oriented topological $n$-manifold, $n\ge2$, let $(N,\rho_N)$ be a closed connected oriented smooth Riemannian $n$-manifold, and let $f\colon M\to N$ be an orientation-preserving branched cover of degree $m\ge1$. Define the \emph{pullback length metric} by
\begin{equation}\label{eq1.1}
 d_f(x,x')=\inf\bigl\{\ell_{\rho_N}(f\circ\gamma)\colon
 \gamma\text{ a path in }M\text{ from }x\text{ to }x'\bigr\}.
\end{equation}
This is the length metric associated with the diameter pullback of Guo and Williams~\cite[Section~5.3, especially Lemma~5.10]{GW2016b}. The distinction between the two metrics and direct proofs of the needed properties are given in Section~\ref{sec4}.

\begin{theorem}\label{thm1.4}
Let $f\colon M\to N$ be as above.
\begin{enumerate}
\item $d_f$ is a compatible geodesic metric, $f\colon(M,d_f)\to(N,\rho_N)$ is $1$-BLD, and $(M,d_f)$ is Ahlfors $n$-regular with constants depending only on $n$, $m$ and $N$.
\item If $d$ is a compatible $C$-quasiconvex metric on $M$ for which $f$ is $L$-BLD, then $L^{-1}d\le d_f\le LC\,d$.
\item $f$ admits a compatible quasiconvex, linearly locally contractible metric for which it is BLD if and only if $(M,d_f)$ is linearly locally contractible. If $M$ is homeomorphic to a smooth manifold, the same holds with ``quasiconvex'' replaced by ``Ahlfors $n$-regular''. In particular, for $M=N=\Sn$, Question~\ref{que1.1} has a positive answer for $f$ if and only if $(\Sn,d_f)$ is linearly locally contractible.
\end{enumerate}
\end{theorem}

Thus Ahlfors regularity is automatic for the canonical metric. The remaining requirement is linear local contractibility, and admissible metrics, when they exist, are unique up to bi-Lipschitz equivalence.

\subsection{Porosity and its insufficiency}
Following~\cite{Luu1998,M1995}, with a common scale threshold made explicit, a closed set $A$ is \emph{uniformly $\delta$-porous at scales below $R_0$} if every $B(a,R)$ with $a\in A$ and $0<R<R_0$ contains a ball $B(y,\delta R)$ disjoint from $A$. Linear local contractibility~\cite{BMW2025,S1996} with \emph{contractibility data} $(\lambda,r_*)$ means that $B(x,r)\hookrightarrow B(x,\lambda r)$ is null-homotopic whenever $r<r_*$.

\begin{theorem}\label{thm1.5}
Let $(M,D)$ be a closed connected oriented topological $n$-manifold with a compatible length metric and contractibility data $(\lambda,r_*)$, let $(N,\rho_N)$ be a closed connected oriented Riemannian $n$-manifold, and let $f\colon(M,D)\to(N,\rho_N)$ be an orientation-preserving $L$-BLD branched cover of degree $m\ge2$. Then:
\begin{enumerate}
\item the branch-value set $f(B_f)$ is uniformly $\delta$-porous in $N$ at small scales, with $\delta=\delta(n,L,\lambda,m)>0$; consequently $\dim_{\mathcal H}f(B_f)\le\dim_Af(B_f)\le n-c(n,L,\lambda,m)<n$ and $\vol_N(f(B_f))=0$;
\item $f$ satisfies the Newman property of Luisto~\cite{L2019} at every branch point, with a constant depending only on $n,L,\lambda,m$, and the branch set $B_f$ is uniformly porous in $(M,D)$ at small scales.
\end{enumerate}
No measure on $M$ is specified as an assumption or used in the obstruction argument. In this finite degree setting, however, Ahlfors regularity of $M$ follows from Theorem~\ref{thm1.4}(i),(ii); thus the absence of an explicit doubling assumption is not an enlargement beyond all doubling source spaces. The scale thresholds depend in addition on $r_*$ and on $N$.
\end{theorem}

The key step is Proposition~\ref{prop5.4}. If a connected component $C$ of the inverse image of a target ball has total degree $d_C\ge2$, then some fiber in $C$ has diameter comparable to the radius. Assuming otherwise, skeletal extension constructs an approximate right inverse on a target cube. Its relative degree must be $1$ in the target and divisible by $d_C$ in the source. Localization at a compact connected inverse image, rather than at one point, gives this common integer coefficient (Lemma~\ref{lem5.3}). A finite recursion then separates the sheets and produces a hole in the entire branch value set.

The cusp construction shows that this necessary condition does not characterize admissibility.

\begin{corollary}\label{cor1.6}
For every $n\ge3$, the degree-two cover $F_n$ of Theorem~\ref{thm1.2} can be chosen so that its branch value set $A_n=F_n(B_{F_n})$ has the following properties:
\begin{enumerate}
\item $A_n$ is ambiently homeomorphic to the standard equator $\mathbb S^{n-2}\subset\Sn$ and is contained in a round great $(n-1)$-sphere;
\item $A_n$ is uniformly $1/4$-porous at all scales below $\pi$, and $\dim_{\mathcal H}A_n=n-2$;
\item $(\Sn,d_{F_n})$ is a compatible Ahlfors $n$-regular geodesic space and $F_n$ is $1$-BLD from it, but this space is not linearly locally contractible.
\end{enumerate}
In particular, uniform porosity of branch values, even together with the dimension and ambient topology in \textup{(i)}--\textup{(ii)}, does not imply admissibility.
\end{corollary}

The geometric assertions are verified in Proposition~\ref{prop6.6}. The metric conclusion combines the independent obstruction in Section~\ref{sec3} with Theorem~\ref{thm1.4}. The obstruction here persists even though the branch value set has zero $n$-dimensional measure.

\subsection{Target changes and the planar case}
Let $q_m\colon\Sn\to\Sn$ be the standard winding cover of degree $m$ branched along the equator $E\cong\mathbb S^{n-2}$, and let $v_n$ denote normalized round volume.

\begin{theorem}\label{thm1.7}
\begin{enumerate}
\item Let $n\ge3$, $m\ge2$ and $0<\theta<1$. There is a homeomorphism $h\colon\Sn\to\Sn$ with $v_n(h(E))=\theta$, and the cover $f_{m,\theta}=h\circ q_m$ admits no compatible quasiconvex, linearly locally contractible metric for which it is BLD. The pullback metric $d_{f_{m,\theta}}$ is a compatible Ahlfors $n$-regular geodesic metric on $\Sn$, and $f_{m,\theta}$ is $1$-BLD from it, but $(\Sn,d_{f_{m,\theta}})$ is not linearly locally contractible.
\item If $q_0\colon\Sn\to\Sn$ admits an admissible metric and $\Phi\colon\Sn\to\Sn$ is orientation-preserving and quasisymmetric, then $\Phi\circ q_0$ admits an admissible metric. In particular $\Phi\circ q_m$ is admissible for every orientation-preserving quasiconformal $\Phi$. In dimension three its branch-value set $\Phi(E)$ may have Hausdorff dimension larger than one.
\item For $n=2$ every orientation-preserving branched cover $\mathbb S^2\to\mathbb S^2$ admits an admissible metric, namely $d_f$.
\end{enumerate}
\end{theorem}

Part~\textup{(i)} recovers the earlier positive-volume obstruction by combining Theorem~\ref{thm1.5} with the Oxtoby--Ulam theorem. Parts~\textup{(ii)}--\textup{(iii)} supply positive results. In particular, the cusp examples and the admissible quasiconformal images of the standard branch locus show that neither ambient flatness nor Hausdorff dimension alone decides the source metric problem. Quasisymmetric invariance is a sufficient preservation principle, not a characterization of all target homeomorphisms preserving admissibility.

\subsection{The dependence on the degree}
The porosity constant in Theorem~\ref{thm1.5} depends on $m$. The next result concerns a scale threshold controlled by the listed geometric data, rather than a threshold chosen separately from the separation of the branch values of each individual map.

\begin{theorem}\label{thm1.8}
\begin{enumerate}
\item For every $n\ge2$ there are constants $\lambda_0,r_0,R_0,C_0>0$ and a fixed closed Riemannian target $N_n$ with the following property. For every $m\ge2$ and every $0<R<R_0$ there is an orientation-preserving $1$-BLD branched cover $F_m\colon M_m\to N_n$ of degree $m$, where $M_m$ is a closed geodesic $n$-manifold with contractibility data $(\lambda_0,r_0)$, and there is $y_0\in F_m(B_{F_m})$ such that every ball of radius greater than $C_0m^{-1/2}R$ contained in $B(y_0,R)$ meets $F_m(B_{F_m})$. The radius $R$ may be chosen below the threshold in Theorem~\ref{thm6.1}. In dimension two both manifolds can be spheres; for $n>2$ we use products with a flat $(n-2)$-torus.
\item For every branched cover $\mathbb S^2\to\mathbb S^2$ of degree $m\ge2$, the branch-value set is uniformly $cm^{-1/2}$-porous below a fixed small round scale, without metric assumptions on the source. Consequently the optimal degree dependence for this sphere-to-sphere class, on a scale uniform over that class, is $m^{-1/2}$.
\end{enumerate}
\end{theorem}

The examples are obtained by gluing sheets along slits with successive transpositions. Their uniform contractibility follows from the disk structure of small inverse image components (Lemma~\ref{lem8.4}) and a flat comparison metric (Proposition~\ref{prop8.5}). The sharp rate in part~\textup{(ii)} is asserted only for covers of the two-sphere by the two-sphere.

\subsection{Relation to earlier work}
The topological study of open discrete maps goes back to Sto\"ilow~\cite{St1928}, Church and Hemmingsen~\cite{CH1960}, \v Cernavski\u\i~\cite{C1964} and V\"ais\"al\"a~\cite{V1966}; path lifting is due to Floyd~\cite{F1950} and Rickman~\cite{R1973}, and~\cite{R1993} is our reference for normal domains and local index. BLD maps were introduced by Martio and V\"ais\"al\"a~\cite{MV1988}. Their metric theory was developed by Heinonen and Rickman~\cite{HR2002}, with subsequent work in~\cite{HM2015,HMZ2018,L2016,L2017,L2019}. Dimension and measure of branch sets are studied in~\cite{BH2004,HR1998,KTW2005,MR1973,Sa1975}, and Newman-type fiber bounds in~\cite{L2019,MR1983,MR1984,N1931}.

There is substantial overlap between the pullback and porosity portions of this paper and Guo--Williams. Their pullback paper~\cite[Lemmas~5.7, 5.8 and 5.10, Proposition~5.13]{GW2016b} already supplies the metric factorization, regularity and comparison principles behind Theorem~\ref{thm1.4}; its Theorem~8.6 treats positive measure branch values and failure of linear local contractibility. These qualitative conclusions are not claimed as new. Skeletal extension and approximate homotopy-inverse arguments occur in~\cite[Lemma~4.5, Theorem~4.7 and Corollary~4.8]{GW2017}; Lemma~\ref{lem5.1} is an explicit instance of that extension principle. Proposition~\ref{prop5.4} isolates the total degree of a connected inverse-image component, and Theorem~\ref{thm6.1} gives a uniform finite-degree conclusion at a common scale, to be compared with the countably porous decomposition in~\cite[Theorem~1.1]{GW2017}. Explicit constants and avoidance of analytic modulus arguments are already features of that earlier work. See also~\cite[Remark~6.7(b)]{HR2002} for the related zero-measure conclusion.

The homological arguments in Sections~\ref{sec3} and~\ref{sec5} have different roles. Section~\ref{sec3} exhibits nonzero classes in $H_1$ and excludes controlled contractions of short loops. Section~\ref{sec5} uses $H_n$ of pairs and total degree to force a large fiber under a contractibility hypothesis. Only the first argument is used for the explicit degree-two counterexample and the negative answer in Corollary~\ref{cor1.3}. Combining it with the canonical metric theory gives Corollary~\ref{cor1.6}.

The paragraph preceding~\cite[Corollary~3.3]{GW2016a} formulates the measure argument in terms of the source branch set $B_f$. Positive round measure of $B_f$ alone does not obstruct remetrization of the source: if $h(E)$ has positive round volume, the map $q_m\circ h^{-1}$ has source branch set $h(E)$ and admits the metric transported by $h$ from the round source of $q_m$. Section~1.2 of~\cite{GW2017} formulates the obstruction using positive measure of $f(B_f)$ in the fixed target and states the negative answer in Corollary~1.6. The survey~\cite[Section~8]{DPT2026} lists Question~27 as open. Luisto~\cite[Introduction]{L2019} discusses the Guo--Williams results. We retain the attribution of the negative answer to Guo and Williams and distinguish it from the explicit homological construction and quantitative results developed here.

The application of the target measure obstruction in~\cite{GW2017} can be checked directly. An admissible metric on $\Sn$ is bi-Lipschitz equivalent to a length metric $D$ by Section~\ref{sec2}. Lifting a reversed short target geodesic in a normal neighborhood gives, for all $x'$ sufficiently close to $x$,
\[
 L^{-1}D(x,x')\le\rho(f(x),f(x'))\le LD(x,x').
\]
Thus the infinitesimal linear dilatation is at most $L^2$, and the doubling and local connectivity assumptions of~\cite[Corollary~1.2]{GW2017} hold, yielding $\vol(f(B_f))=0$. Sections~\ref{sec5}--\ref{sec7} give direct proofs of the needed statements in the present finite-degree setting.

The analytic input used to pass from Ahlfors regularity and linear local contractibility to quasiconvexity is Semmes' Poincar\'e theorem~\cite{S1996}, in the form of Basso, Marti and Wenger~\cite[Corollary~1.6]{BMW2025}, together with~\cite{BB2018,HK1998}. The cusp argument otherwise uses elementary covering space topology and the explicit annulus calculation. The positive volume examples use the Oxtoby--Ulam theorem~\cite{Fa1980,LR2014,OU1941}; the quasisymmetric arguments use ball distortion~\cite{H2001}, and the fractal examples use Ahlfors' extension theorem~\cite{A1964}, generalized to higher dimensions by Tukia and V\"ais\"al\"a~\cite{TV1982}.

\subsection{Organization}
Section~\ref{sec2} fixes conventions and proves the reductions to quasiconvex and length metrics. Section~\ref{sec3} proves Theorem~\ref{thm1.2} and Corollary~\ref{cor1.3}. Section~\ref{sec4} establishes the canonical criterion in Theorem~\ref{thm1.4}. Sections~\ref{sec5}--\ref{sec6} prove the quantitative porosity theorem and complete the connection with the cusp examples in Corollary~\ref{cor1.6}. Section~\ref{sec7} treats target changes and dimension two, and Section~\ref{sec8} proves the degree estimates. Section~\ref{sec9} formulates open questions on degree dependence, the geometry of branch values, and target changes beyond quasisymmetry.

\section{Preliminaries and metric reductions}\label{sec2}

\subsection{Metric conventions}
All metric spaces are separable and all manifold metrics are compatible. Balls are open unless stated otherwise. $\HH^s_d$ denotes $s$-dimensional Hausdorff measure for a metric $d$. For a path $\gamma\colon[a,b]\to X$,
\begin{equation}\label{eq2.1}
 \ell_d(\gamma)=\sup_{a=t_0<\cdots<t_N=b}\sum_{j=1}^N d(\gamma(t_{j-1}),\gamma(t_j))\in[0,\infty].
\end{equation}
A metric $D$ is a \emph{length metric} if distances are infima of lengths of joining paths; a space is \emph{$C$-quasiconvex} if any two points are joined by a path of length at most $C$ times their distance. A compact length space is geodesic~\cite{BBI2001}. A metric measure space $(X,d,\mu)$ is \emph{Ahlfors $n$-regular} if $C_A^{-1}r^n\le\mu(B(x,r))\le C_Ar^n$ for $x\in X$, $0<r\le\diam X$; a metric space is Ahlfors $n$-regular when this holds for $\mu=\HH^n_d$. See~\cite{H2001,HKST2015}.

\subsection{Linear local contractibility and porosity}
\begin{definition}\label{def2.1}
Following~\cite{BMW2025,S1996}, $(X,d)$ is \emph{linearly locally contractible} with contractibility data $(\lambda,r_*)$, where $\lambda\ge1$ and $r_*>0$, if for every $x\in X$ and $0<r<r_*$ the inclusion $B(x,r)\hookrightarrow B(x,\lambda r)$ is null-homotopic in $B(x,\lambda r)$.
\end{definition}
The homotopies need only be continuous. Every continuous map from a sphere into $B(x,r)$ then extends to its cone with image in $B(x,\lambda r)$; this is applied to boundaries of simplices, and the source manifold is never triangulated.

Using the porosity convention of~\cite{Luu1998,M1995}, a closed set $A$ is \emph{uniformly $\delta$-porous at scales below $R_0$} if for every $a\in A$ and $0<R<R_0$ there is $B(y,\delta R)\subset B(a,R)\setminus A$. A uniformly porous subset of a Riemannian manifold has Assouad dimension strictly less than $n$~\cite[Theorem~5.2]{Luu1998} and, in particular, volume zero~\cite[Section~11]{M1995}.

\subsection{Branched covers}
A \emph{branched cover} is a continuous, open, discrete map; between oriented manifolds it is assumed orientation-preserving. Its branch set is
\[
 B_f=\{x\colon f\text{ is not a local homeomorphism at }x\};
\]
$f(B_f)$ is the branch-value set. For a branched cover between closed connected oriented $n$-manifolds the degree $m=\deg f\ge1$ is defined by $f_*[M]=m[N]$. A relatively compact domain $U\subset M$ is a \emph{normal domain} if $f(\partial U)=\partial f(U)$; it is a \emph{normal neighborhood} of $x$ if in addition $\overline U\cap f^{-1}(f(x))=\{x\}$. The local index $i(x,f)\ge1$ is the local degree on a small normal neighborhood, and
\begin{equation}\label{eq2.2}
 \sum_{x\in f^{-1}(y)}i(x,f)=\deg f.
\end{equation}
For a normal domain $U$ the sum $\mu(y,f,U)=\sum_{x\in U\cap f^{-1}(y)}i(x,f)$ is independent of $y\in f(U)$; we call it the \emph{degree of $U$}. If $U$ is a normal domain and $V\subset f(U)$ is a domain, every component of $U\cap f^{-1}(V)$ is a normal domain mapped onto $V$. These facts are proved in~\cite[Chapter~I]{R1993} and~\cite[Section~5]{V1966} using only openness, discreteness and properness, and they hold for topological manifolds. A point $x$ with $i(x,f)=1$ is a regular point, and $x\in B_f$ if and only if $i(x,f)\ge2$.

\subsection{BLD maps}
Following~\cite{HR2002,L2017,MV1988}, a branched cover $f\colon(X,d_X)\to(Y,d_Y)$ is \emph{$L$-BLD} if
\begin{equation}\label{eq2.3}
 L^{-1}\ell_{d_X}(\gamma)\le\ell_{d_Y}(f\circ\gamma)\le L\ell_{d_X}(\gamma)
\end{equation}
for every continuous path $\gamma$, lengths being allowed to be infinite. The weaker notion requiring~\eqref{eq2.3} only for rectifiable paths~\cite{HM2015,HMZ2018} is not used.

For a source metric $D$, following~\cite[Introduction]{L2019}, $f$ has the \emph{Newman property at $x_0$} if there are a neighborhood $U$ of $x_0$ and $\varepsilon>0$ such that every ball $B_D(z,t)\subset U$ with $z\in B_f$ contains a point $w$ satisfying
\[
 \diam_D\bigl(B_D(z,t)\cap f^{-1}(f(w))\bigr)\ge\varepsilon t.
\]
The centers of the tested balls must be branch points. The neighborhood base point $x_0$ is arbitrary in Luisto's definition, so the property is vacuous on a sufficiently small neighborhood of a regular point; no positive fiber-diameter estimate is asserted for balls centered at regular points.

\subsection{Homology}
Homology has integer coefficients. For an oriented $n$-manifold $M$ and $x\in M$, $H_n(M,M\setminus\{x\})\cong\Z$ with positive generator $o_x$; if $M$ is closed and connected, $[M]$ maps to $o_x$ and $H_n(M)\to H_n(M,M\setminus\{x\})$ is an isomorphism~\cite[Theorem~3.26]{Ha2002}. To define the local index, choose a normal neighborhood $U$ with $U\cap f^{-1}(f(x))=\{x\}$. Then the map of pairs $f\colon(U,U\setminus\{x\})\to(N,N\setminus\{f(x)\})$ induces multiplication by $i(x,f)$ on the local orientation generators~\cite[Proposition~2.30]{Ha2002}\cite[Chapter~I]{R1993}. The neighborhood restriction is essential: on all of $M$, other points of $f^{-1}(f(x))$ would prevent this from being a map of pairs.

\subsection{Quasisymmetric maps}
A homeomorphism $\Phi\colon X\to Y$ between metric spaces is \emph{$\eta$-quasisymmetric} if $|x-a|\le t|x-b|$ implies $|\Phi(x)-\Phi(a)|\le\eta(t)|\Phi(x)-\Phi(b)|$, for a homeomorphism $\eta\colon[0,\infty)\to[0,\infty)$. Quasiconformal self-maps of $\Sn$ are quasisymmetric and preserve null sets~\cite{H2001,V1971}.

\subsection{Quasiconvexity and intrinsic length metrics}

\begin{theorem}\label{thm2.2}
Let $X$ be a metric space homeomorphic to a closed orientable smooth $n$-manifold, $n\ge2$. If $X$ is Ahlfors $n$-regular and linearly locally contractible, then $(X,d,\HH^n_d)$ supports a weak $(1,1)$-Poincar\'e inequality.
\end{theorem}
This is~\cite[Corollary~1.6]{BMW2025}, which reproves the theorem of Semmes~\cite{S1996}. Its topological hypotheses are exactly that $X$ be homeomorphic to a closed (compact without boundary), orientable, smooth manifold of dimension $n\ge2$; the compatible metric need not be smooth. Their convention requires contractions for $r<\diam X/\Lambda$. Definition~\ref{def2.1} implies this convention after choosing $\Lambda\ge\max\{\lambda,2\diam X/r_*\}$. Thus Theorem~\ref{thm2.2} applies to the sphere metrics used in Corollary~\ref{cor1.3}.

\begin{corollary}\label{cor2.3}
Every compatible Ahlfors $n$-regular, linearly locally contractible metric on a closed connected oriented manifold homeomorphic to a smooth manifold is quasiconvex.
\end{corollary}
\begin{proof}
The space is compact, connected and complete, its Hausdorff measure is doubling, and Theorem~\ref{thm2.2} gives a Poincar\'e inequality. Quasiconvexity follows from~\cite[Proposition~4.8 and Lemma~4.9]{BB2018}; a ball of radius larger than the diameter gives one global constant.
\end{proof}

\begin{lemma}\label{lem2.4}
Let $(X,d)$ be compact and $C$-quasiconvex, and let
\[
 D(x,x')=\inf\{\ell_d(\gamma)\colon\gamma\text{ joins }x,x'\}.
\]
Then $D$ is a compatible compact length metric, $d\le D\le Cd$, and $\ell_D(\gamma)=\ell_d(\gamma)$ for every path. If $d$ has contractibility data $(\lambda_0,r_*)$ then $D$ has data $(C\lambda_0,r_*)$. Replacing $d$ by $D$ changes no BLD constant.
\end{lemma}
\begin{proof}
The bounds $d\le D\le Cd$ are the definition of length and quasiconvexity; the triangle inequality follows by concatenation, and the topology is unchanged. Since $D\ge d$, $\ell_D\ge\ell_d$; conversely, for $\ell_d(\gamma)<\infty$ and any partition, $\sum D(\gamma(t_{j-1}),\gamma(t_j))\le\sum\ell_d(\gamma|_{[t_{j-1},t_j]})=\ell_d(\gamma)$. Hence $\ell_D=\ell_d$, so $D$ is a length metric. For contractibility use $B_D(x,s)\subset B_d(x,s)$ and $B_d(x,\lambda_0s)\subset B_D(x,C\lambda_0s)$.
\end{proof}

\subsection{Path lifting and length control}
\begin{lemma}\label{lem2.5}
Let $f\colon(X,D)\to(Y,d_Y)$ be an $L$-BLD map between complete locally compact length spaces. Every rectifiable path $\beta\colon[0,1]\to Y$ has a full lift $\alpha$ starting at any prescribed $x\in f^{-1}(\beta(0))$, and $\ell_D(\alpha)\le L\ell_{d_Y}(\beta)$. The map $f$ is $L$-Lipschitz.
\end{lemma}
\begin{proof}
We first recall the topological input: a branched cover between locally compact complete length spaces has local path lifts in normal neighborhoods~\cite[Proposition~2.6]{L2017}; this is the local form of the path-lifting results of~\cite{F1950,R1973}. We explain why a rectifiable path has a full lift and why no rectifiability of the lift needs to be assumed in advance.

Starting from $x$, extend a local lift as far as possible, obtaining a maximal lift $\alpha\colon[0,T)\to X$, $0<T\le1$. Such a maximal lift can be obtained by ordering extensions of the initial lift by restriction; the union of a chain is again a lift on the union of its parameter intervals. For $0\le u<v<T$, the lower BLD inequality, which applies to every continuous path, gives
\[
 \ell_D(\alpha|_{[u,v]})
 \le L\ell_{d_Y}(\beta|_{[u,v]}),\qquad
 D(\alpha(u),\alpha(v))\le L\ell_{d_Y}(\beta|_{[u,v]}).
\]
The length function of a continuous rectifiable path is continuous. Hence the right-hand side tends to zero as $u,v\uparrow T$. Completeness of $X$ supplies a limit $x_T$, and continuity of $f$ gives $f(x_T)=\beta(T)$. Defining $\alpha(T)=x_T$ extends the lift continuously. If $T<1$, a local lift of the remaining path starting at $x_T$ extends it further, contrary to maximality. Thus $T=1$ and the extension is a full lift. Applying the same BLD inequality on $[0,1]$ proves $\ell_D(\alpha)\le L\ell_{d_Y}(\beta)$. This also explains the completeness step in~\cite[Lemma~2.8]{L2017}.

Finally, for $u,v\in X$ and $\varepsilon>0$, choose a path $\gamma$ joining them with $\ell_D(\gamma)<D(u,v)+\varepsilon$. Then
\[
 d_Y(f(u),f(v))\le\ell_{d_Y}(f\circ\gamma)
 \le L\ell_D(\gamma)<L(D(u,v)+\varepsilon).
\]
Letting $\varepsilon\downarrow0$ proves that $f$ is $L$-Lipschitz.
\end{proof}

\section{Essential loops and a cuspidal obstruction}\label{sec3}

The obstruction in this section uses only the BLD inequalities, quasiconvexity, a one-dimensional homological form of linear local contractibility, and the topology of an explicitly computed local double cover. In particular, the canonical metric of Section~\ref{sec4} is not used.

\subsection{A metric obstruction from short essential loops}

We say that $(X,d)$ satisfies \emph{one-dimensional homological linear local contractibility} with data $(\lambda,r_*)$ if $\lambda\ge1$, $r_*>0$, and the inclusion-induced map
\[
 H_1(B_d(x,r);\Z)\longrightarrow H_1(B_d(x,\lambda r);\Z)
\]
is zero for every $x\in X$ and $0<r<r_*$. Equivalently, every continuous loop in $B_d(x,r)$ bounds a singular $2$-chain in $B_d(x,\lambda r)$; the equivalence follows from the description of $H_1$ by loop classes on each path component~\cite[Theorem~2A.1]{Ha2002}. Linear local contractibility implies this condition, but no null-homotopy is required here.

\begin{lemma}\label{lem3.1}
Let $f\colon X\to(Y,\rho)$ be a fixed continuous map. Suppose that there are continuous loops $\gamma_j\colon\mathbb S^1\to X$, points $x_j$ on these loops, open sets $U_j\subset Y$, and numbers $b_j>0$ such that
\begin{enumerate}
\item $B_\rho(f(x_j),b_j)\subset U_j$ and $\gamma_j(\mathbb S^1)\subset f^{-1}(U_j)$;
\item $[\gamma_j]\ne0$ in $H_1(f^{-1}(U_j);\Z)$;
\item $0<\ell_j:=\ell_\rho(f\circ\gamma_j)<\infty$, $\ell_j\to0$, and $\ell_j/b_j\to0$.
\end{enumerate}
Then no compatible metric $d$ on $X$ is simultaneously quasiconvex, satisfies one-dimensional homological linear local contractibility, and makes $f$ BLD. In particular, there is no compatible quasiconvex, linearly locally contractible BLD source metric.
\end{lemma}
\begin{proof}
Suppose that such a metric exists, and let $C,L,\lambda,r_*$ be its quasiconvexity, BLD, and homological contractibility constants. First, quasiconvexity and the upper BLD inequality imply
\begin{equation}\label{eq3.1}
 \rho(f(u),f(v))\leq LC\,d(u,v)\qquad(u,v\in X).
\end{equation}
Indeed, apply the upper length inequality to a quasiconvex path joining $u$ to $v$.

The lower BLD inequality, which applies to all continuous paths, gives
\begin{equation}\label{eq3.2}
 \ell_d(\gamma_j)\leq L\ell_j.
\end{equation}
Thus the rectifiability of the loops in the unknown source metric is a consequence, not an extra assumption. Set $r_j=2L\ell_j$. Parameterizing $\gamma_j$ from $x_j$, every point of the loop has distance at most $L\ell_j<r_j$ from $x_j$. For all sufficiently large $j$, we have
\[
 r_j<r_*\quad\hbox{and}\quad LC\lambda r_j=2L^2C\lambda\ell_j<b_j.
\]
Homological linear local contractibility makes $[\gamma_j]=0$ in $H_1(B_d(x_j,\lambda r_j);\Z)$. By~\eqref{eq3.1},
\[
 B_d(x_j,\lambda r_j)\subset
 f^{-1}\bigl(B_\rho(f(x_j),b_j)\bigr)\subset f^{-1}(U_j).
\]
The inclusion-induced map therefore sends $[\gamma_j]$ to zero in $H_1(f^{-1}(U_j);\Z)$, contrary to (ii).
\end{proof}

\subsection{The local double cover calculation}

For $R>0$ let $D_R=\{w\in\C:|w|<R\}$. We recall the needed covering-space calculation in detail. Throughout this section, \emph{simple branching} means topological equivalence, in local homeomorphism charts on both source and target, to
\[
 (z,v)\longmapsto(z^2,v),\qquad z\in\C,\quad v\in\R^{n-2}.
\]
The charts are not assumed to be smooth, Lipschitz, or length preserving. This topological notion of simple branching is preserved by postcomposition with a target homeomorphism.

\begin{lemma}\label{lem3.2}
Let $I$ be an open interval, let $0<a<R$, and let
\[
 p\colon W\longrightarrow I\times D_R
\]
be a proper branched double cover whose branch-value set is $I\times\{-a,a\}$, with simple branching along both components. Assume that each branch value has its unique ramification preimage. Then, as a branched cover, $p$ is equivalent to
\[
 \id_I\times\pi\colon I\times\Sigma\longrightarrow I\times D_R,
\]
where $\Sigma$ is an open annulus and $\pi$ is the double cover of the disk branched at $-a,a$. For every $s_0\in I$, the inverse image of
\[
 \{s_0\}\times[-a,a]
\]
is a simple closed curve representing a generator, up to sign, of $H_1(W;\Z)\cong\Z$.
\end{lemma}
\begin{proof}
An explicit model for the disk cover is
\[
 \Sigma=\{(w,\zeta)\in D_R\times\C:\zeta^2=w^2-a^2\},
 \qquad \pi(w,\zeta)=w.
\]
It is a smooth surface: the two complex partial derivatives of $\zeta^2-w^2+a^2$ cannot both vanish on its zero set, since $a>0$. Projection is a local homeomorphism when $\zeta\ne0$, and near $\zeta=0$, $w=\pm a$, it has a simple square-map branch. The map is proper, as both $w$ and $\zeta$ are bounded over compact subsets of $D_R$. There is exactly one preimage of each of $-a,a$ and two preimages of every other point.

The complement of the two branch-value arcs in $I\times D_R$ has fundamental group freely generated by meridians about the two punctures. To determine their actions, choose a small target product neighborhood $V$ about one branch value and the normal neighborhood $U$ of its unique ramification preimage supplied by the local topological square-map model. The set $U\setminus B_p$ is connected: in these charts it is the product of a punctured disk with an interval. Consequently
\[
 p\colon U\setminus B_p\longrightarrow V\setminus p(B_p)
\]
is a connected two-sheeted covering. The branch-value arc in $V$ is tame, so $\pi_1(V\setminus p(B_p))\cong\Z$, generated by a meridian. Connectedness makes its action transitive on the two sheets; hence the meridian acts by the nontrivial transposition. The same argument applies to the other arc. Thus the global monodromy is transitive and specifies a unique connected double cover of the complement, with exactly the monodromy of $\id_I\times\pi$. The classification of covering spaces, for example~\cite[Theorem~1.38]{Ha2002}, gives an equivalence away from the branch values.

This equivalence extends over the branch locus. One can see this in the local square-map model: the two sheets merge into the unique point above the branch value. Equivalently, properness and the uniqueness of that preimage force every lifted sequence approaching a branch value to converge to that point. Applying the same argument to the inverse equivalence proves that the extension is a homeomorphism. Consequently it suffices to compute the disk cover.

We calculate the topology of $\Sigma$ without invoking a surface-classification theorem. The identity
\[
 (w+\zeta)(w-\zeta)=a^2
\]
shows that $u=(w+\zeta)/a$ is nonzero. It is a coordinate with inverse
\[
 w=\frac a2(u+u^{-1}),\qquad
 \zeta=\frac a2(u-u^{-1}).
\]
Writing $u=e^{s+i\theta}$, the constraint $|w|<R$ becomes
\[
 a^2(\sinh^2s+\cos^2\theta)<R^2,
 \quad\text{or}\quad
 |s|<b(\theta):=\operatorname{arsinh}
       \sqrt{(R/a)^2-\cos^2\theta}.
\]
The function $b$ is positive and continuous because $R>a$. Thus $(\theta,s)\mapsto(\theta,s/b(\theta))$ identifies $\Sigma$ with the open annulus $\mathbb S^1\times(-1,1)$. The homotopy $(\theta,s)\mapsto(\theta,(1-\tau)s)$, $0\leq\tau\leq1$, stays in $\Sigma$ and is a strong deformation retraction onto $s=0$.

This core circle is exactly
\[
 (w,\zeta)=(a\cos\theta,ia\sin\theta),\qquad
 \pi^{-1}([-a,a])=\{(a\cos\theta,ia\sin\theta):\theta\in\R/2\pi\Z\}.
\]
It therefore generates $H_1(\Sigma;\Z)\cong\Z$. As a direct check, the continuous map $(w,\zeta)\mapsto(w+\zeta)/|w+\zeta|$ sends this circle to $e^{i\theta}$ with winding number one, so the circle cannot be null-homotopic. Taking the product with $I$ proves the lemma.
\end{proof}

The same conclusion about the homology class holds after taking a product with any ball. This will be used in higher dimensions; taking such a product does not contract the essential circle.

\subsection{An explicit cover in dimension three}

Let $q\colon\mathbb S^3\to\mathbb S^3$ be the standard orientation-preserving branched double cover. One formula, on $\mathbb S^3\subset\C\times\R^2$, is
\[
 q(z,u)=
 \begin{cases}
 (z^2/|z|,u),&z\ne0,\\
 (0,u),&z=0.
 \end{cases}
\]
It is open, has degree two, and has simple branching along the round circle $E=\{z=0\}$. The angular doubling map in this formula is topologically equivalent to the square map by a radial homeomorphism in the transverse disk.

Choose oriented stereographic coordinates $\psi\colon\R^3\to\mathbb S^3\setminus\{p_\infty\}$ in which the branch-value circle is
\[
 C_0=\{(x,y,0):(x-1)^2+y^2=1\}.
\]
Choose the stereographic pole to be a unit vector orthogonal to the two-dimensional linear subspace containing $E$. Stereographic projection then sends $E$ to a round planar circle; an oriented Euclidean similarity within that plane puts it in the displayed position. With this choice the plane $\{z=0\}$, together with the pole, corresponds to a round great two-sphere $G_3\subset\mathbb S^3$. We retain this choice when proving the porosity assertion in Proposition~\ref{prop6.6}. The round metric is smooth and uniformly comparable to the Euclidean metric on each fixed compact coordinate neighborhood used below.

Here is a completely specified homeomorphism producing a cusp. Put $c=(1,0)$ and define a radial homeomorphism $T\colon\R^2\to\R^2$ by
\[
 T(c)=c,\qquad
 T(c+v)=c+\frac{\|v\|_2}{\|v\|_1}v\quad(v\ne0).
\]
Its inverse is
\[
 T^{-1}(c)=c,\qquad
 T^{-1}(c+w)=c+\frac{\|w\|_1}{\|w\|_2}w\quad(w\ne0).
\]
It maps the circle centered at $c$ of radius one onto the diamond $|x-1|+|y|=1$. Define also
\[
 S(x,y)=(x,y|y|),\qquad
 S^{-1}(x,y)=(x,\operatorname{sgn}(y)\sqrt{|y|}).
\]
Both are proper orientation-preserving homeomorphisms. For $T$ this follows also by interpolating its positive radial factor with $1$; for $S$ it follows from the strict increase of $y\mapsto y|y|$. Extend $S\circ T$ to $\R^3$ by leaving the third coordinate unchanged. This is a proper orientation-preserving homeomorphism $H$ and therefore extends over the point at infinity to a homeomorphism $h$ of $\mathbb S^3$ in the coordinates $\psi$.

Set
\begin{equation}\label{eq3.3}
 f_3=h\circ q.
\end{equation}
This is an orientation-preserving branched double cover, with branch-value curve $\psi(\Gamma)$, where $\Gamma=H(C_0)$. In the slab $0<x<1$ it has exactly the form
\begin{equation}\label{eq3.4}
 \Gamma\cap\{0<x<1\}
 =\{(s,s^2,0):0<s<1\}\cup\{(s,-s^2,0):0<s<1\}.
\end{equation}
In particular the two arcs form a cusp at the origin. The curve is an ambiently unknotted Jordan curve because it is the image of $C_0$ under the explicitly given ambient homeomorphism $H$. No limiting construction or assertion about tameness is needed.

For $0<t<1/64$ put
\[
 I_t=(t/2,3t/2),\qquad R_t=t/4,
\]
and define
\begin{equation}\label{eq3.5}
 \Phi_t(s,w)=\left(s,\frac{s^2}{t^2}\Re w,
                         \frac{s^2}{t^2}\Im w\right),\qquad
 V_t=\Phi_t(I_t\times D_{R_t}),\qquad U_t=\psi(V_t).
\end{equation}
This is a diffeomorphism onto its image. Since $s>0$, equation~\eqref{eq3.4} gives the exact identity
\begin{equation}\label{eq3.6}
 \Phi_t^{-1}(\Gamma\cap V_t)=I_t\times\{-t^2,t^2\}.
\end{equation}
Here $t^2<R_t$. Define the segment and its endpoint by
\[
 J_t=\{(t,y,0):-t^2\leq y\leq t^2\},\qquad
 a_t=(t,t^2,0).
\]

\begin{lemma}\label{lem3.3}
There are constants $A,c_0>0$, independent of sufficiently small $t$, and a simple closed curve $\gamma_t$ in $f_3^{-1}(U_t)$ such that
\begin{enumerate}
\item $B_{\rho_3}(\psi(a_t),c_0t)\subset U_t$;
\item $\gamma_t(\mathbb S^1)=f_3^{-1}(\psi(J_t))$, and $[\gamma_t]$ generates $H_1(f_3^{-1}(U_t);\Z)\cong\Z$;
\item $0<\ell_{\rho_3}(f_3\circ\gamma_t)\leq A t^2$.
\end{enumerate}
\end{lemma}
\begin{proof}
We first verify a Euclidean ball inclusion with a uniform constant:
\begin{equation}\label{eq3.7}
 B_{\mathrm{Euc}}(a_t,t/64)\subset V_t.
\end{equation}
Indeed, for $(s,y,z)$ in this ball we have $s\in I_t$ and
\[
 \sqrt{y^2+z^2}\leq t^2+t/64\leq t/32.
\]
At any $s\in I_t$ the transverse disk defining $V_t$ has radius
\[
 \frac{s^2}{t^2}R_t=\frac{s^2}{4t}>t/16.
\]
These inequalities prove~\eqref{eq3.7}. The centers and domains stay in a fixed small coordinate neighborhood. Uniform comparison with the round distance, and a smaller constant if necessary to stay within that chart, gives (i).

The restriction of $f_3=h\circ q$ over $U_t$ is proper. Its degree is two, and every branch value has its unique ramification preimage. It also has simple branching in the topological sense fixed above: if $\alpha$ and $\beta$ are source and target charts with $\beta\circ q\circ\alpha^{-1}(z,v)=(z^2,v)$, then
\[
 (\beta\circ h^{-1})\circ f_3\circ\alpha^{-1}(z,v)=(z^2,v).
\]
Thus postcomposition preserves the topological square-map model, without requiring $h$ to preserve any metric or smooth structure. In the coordinates~\eqref{eq3.5}, the branch values are exactly the two product arcs in~\eqref{eq3.6}. Their small local complements have fundamental group $\Z$, and the connected punctured normal neighborhoods give the meridian transpositions established in Lemma~\ref{lem3.2}. That lemma, with $a=t^2$, identifies the full inverse image of $U_t$ with an annulus times $I_t$. It also proves (ii), including that the \emph{full} inverse image of the segment is the essential core circle.

Parameterize that circle by traversing one lift of $\psi(J_t)$ from one ramification point to the other and returning along the other lift. The projection traverses the segment twice. If $M\geq1$ bounds the round norm of tangent vectors by $M$ times their Euclidean norm on the fixed coordinate neighborhood, then
\[
 \ell_{\rho_3}(f_3\circ\gamma_t)
 =2\ell_{\rho_3}(\psi(J_t))\leq4Mt^2.
\]
This proves (iii) with $A=4M$. The projected length is positive since the endpoints of the segment are distinct.
\end{proof}

\begin{proposition}\label{prop3.4}
The map $f_3$ in~\eqref{eq3.3} satisfies Theorem~\ref{thm1.2} in dimension three.
\end{proposition}
\begin{proof}
Take any sequence $t_j\downarrow0$ in $(0,1/64)$, use the loops and neighborhoods from Lemma~\ref{lem3.3}, and choose $x_j$ to be the unique point above $\psi(a_{t_j})$. It lies on $\gamma_{t_j}$. With $b_j=c_0t_j$ and $\ell_j=\ell_{\rho_3}(f_3\circ\gamma_{t_j})$, we have
\[
 0<\ell_j\leq A t_j^2\longrightarrow0,
 \qquad \frac{\ell_j}{b_j}\leq\frac{A}{c_0}t_j\longrightarrow0.
\]
All the hypotheses of Lemma~\ref{lem3.1} are verified, so its conclusion gives the proposition.
\end{proof}

\subsection{Suspension and completion of the proof}

Let $n>3$ and $k=n-3$. Regard the unit sphere $\Sn$ as a subset of $\R^4\times\R^k$. Every point outside $\{0\}\times\mathbb S^{k-1}$ can be written uniquely as
\[
 (\sqrt{1-|v|^2}\,u,v),\qquad u\in\mathbb S^3,\quad |v|<1.
\]
Define
\begin{equation}\label{eq3.8}
 F_n(\sqrt{1-|v|^2}\,u,v)
 = (\sqrt{1-|v|^2}\,f_3(u),v),
 \qquad F_n(0,v)=(0,v)\quad(|v|=1).
\end{equation}
This is the iterated suspension of $f_3$. We verify the global properties explicitly, so that the higher-dimensional assertion is still about spheres rather than products with tori or other manifolds.

\begin{lemma}\label{lem3.5}
The map $F_n$ is an orientation-preserving branched double cover $\Sn\to\Sn$. On the open set $|v|<1$ it is the product $f_3\times\id$ in the displayed coordinates.
\end{lemma}
\begin{proof}
Continuity away from $|v|=1$ is immediate. As $|v|\to1$, the first coordinate of the image has norm $\sqrt{1-|v|^2}\to0$, which proves continuity at the remaining points. Off that set, the product description proves openness and gives fibers of at most two points. At a point $(0,v_0)$, a round ball centered there is characterized by an inequality on $v\cdot v_0$; it imposes no condition on $u$. Formula~\eqref{eq3.8} preserves $v$, and $f_3$ is onto. Consequently it maps every such ball onto itself, which proves openness also at $(0,v_0)$. The fiber over $(0,v_0)$ is a singleton, so the map is discrete everywhere.

Regular values in the product region have two preimages, and both local maps preserve orientation because $f_3$ does. Thus the global degree is two. The map is sense-preserving also at every other point: in a sufficiently small normal neighborhood, its local degree can be computed at nearby regular values, where all contributions have positive sign. This also follows directly from the local product description and the cone description at $|v|=1$.
\end{proof}

Use the smooth target coordinate map
\[
 \Theta(p,v)=(\sqrt{1-|v|^2}\,\psi(p),v)
\]
for $p$ near the origin in $\R^3$ and $|v|<1/4$. For small $t$ define
\[
 U_t^{(n)}=\Theta\bigl(V_t\times B_{\R^k}(0,t)\bigr),
 \qquad y_t^{(n)}=\Theta(a_t,0).
\]
The product formula yields the exact homeomorphism
\begin{equation}\label{eq3.9}
 F_n^{-1}(U_t^{(n)})\cong
 f_3^{-1}(U_t)\times B_{\R^k}(0,t).
\end{equation}
Let $\gamma_t^{(n)}$ be the copy of $\gamma_t$ in the slice $v=0$, and let $x_t^{(n)}$ be its point above $y_t^{(n)}$.

\begin{proof}[Proof of Theorem~\ref{thm1.2}]
For $n=3$ take $F_3=f_3$ and apply Proposition~\ref{prop3.4}. Suppose that $n>3$ and use~\eqref{eq3.8}. Lemma~\ref{lem3.5} verifies the required global covering properties.

The inclusion of the slice $v=0$ in~\eqref{eq3.9} is a homotopy equivalence, with inverse given by projection onto the first factor. Hence
\[
 [\gamma_t^{(n)}]\ne0\quad\hbox{in}\quad
 H_1(F_n^{-1}(U_t^{(n)});\Z).
\]
The target slice $\mathbb S^3\times\{0\}$ is a round great three-sphere in $\Sn$, so it has exactly the same curve lengths as $\mathbb S^3$. Consequently
\[
 0<\ell_{\rho_n}(F_n\circ\gamma_t^{(n)})
 =\ell_{\rho_3}(f_3\circ\gamma_t)\leq A t^2.
\]
By~\eqref{eq3.7}, the coordinate domain $V_t\times B_{\R^k}(0,t)$ contains the Euclidean ball of radius $t/64$ centered at $(a_t,0)$. On a fixed smaller coordinate neighborhood the map $\Theta$ and its inverse have uniformly bounded differentials. Shrinking a constant to stay in that chart, there is $c_n>0$, independent of sufficiently small $t$, such that
\[
 B_{\rho_n}(y_t^{(n)},c_nt)\subset U_t^{(n)}.
\]
Apply Lemma~\ref{lem3.1} to any sequence $t_j\downarrow0$. The ratio of image-loop length to the available target radius is at most $(A/c_n)t_j\to0$. This proves the theorem in every dimension.
\end{proof}

\begin{proof}[Proof of Corollary~\ref{cor1.3}]
If one of the maps $F_n$ admitted a source metric as in Question~27, Corollary~\ref{cor2.3} would make that metric quasiconvex. It would then have all three properties excluded by Theorem~\ref{thm1.2}. This contradiction proves the corollary.
\end{proof}

The scale mismatch in the construction is
\[
 \ell_{\rho_n}(F_n\circ\gamma_t^{(n)})\le At^2,\qquad
 B_{\rho_n}(F_n(x_t^{(n)}),c_nt)\subset U_t^{(n)},\qquad
 [\gamma_t^{(n)}]\ne0\text{ in }H_1(F_n^{-1}(U_t^{(n)});\Z),
\]
with the dimension-three notation understood when $n=3$. Any fixed quasiconvexity, BLD, and one-dimensional homological contractibility data would make the loop null-homologous inside the same inverse image for sufficiently small $t$. We next place this obstruction in the canonical metric framework; this is a consequence of the construction, not an input to it.

\section{The pullback length metric}\label{sec4}

In this section $f\colon M\to N$ is an orientation-preserving branched cover of degree $m$ from a closed connected oriented topological $n$-manifold to a closed connected oriented Riemannian $n$-manifold $(N,\rho_N)$. Choose $r_N>0$ uniformly below the convexity and injectivity radii of $N$, small enough that the normal coordinate charts used below are $2$-bi-Lipschitz and $c_Ns^n\le\vol_N(B(y,s))\le C_Ns^n$ for $s<r_N$. In particular these target balls are strongly convex. We use the normalization of Hausdorff measure for which $\HH^n_{\rho_N}=\vol_N$.

For comparison with~\cite[Section~5]{GW2016b}, set
\[
 d_f^{\rm diam}(x,x')=\inf\{\diam_{\rho_N}f(K):
 K\subset M\text{ a continuum containing }x,x'\}.
\]
Guo and Williams use this diameter pullback. Its associated length metric is exactly $d_f$, by their length-preservation statement and Lemma~5.10. Thus the two constructions have the same path lengths and are bi-Lipschitz equivalent in the finite-degree setting, but their distances need not be identical. We include direct proofs of the facts needed here.

\begin{proposition}\label{prop4.1}
\begin{enumerate}
\item $d_f$ is a metric on $M$ with $d_f\ge\rho_N\circ(f\times f)$, and it induces the manifold topology.
\item $\ell_{d_f}(\gamma)=\ell_{\rho_N}(f\circ\gamma)$ for every path $\gamma$. Hence $(M,d_f)$ is a compact geodesic space and $f\colon(M,d_f)\to(N,\rho_N)$ is $1$-BLD.
\item If $d$ is a compatible $C$-quasiconvex metric on $M$ for which $f$ is $L$-BLD, then $L^{-1}d\le d_f\le LCd$.
\end{enumerate}
\end{proposition}
\begin{proof}
(i) First $d_f$ is finite. At any $z\in M$, take a normal neighborhood whose image is a small strongly convex target ball centered at $f(z)$ and whose fiber over $f(z)$ consists only of $z$. Lifting the reversed radial geodesic from any image point to $f(z)$ gives a path to $z$ with finite image length; proper path lifting is applicable here by~\cite[Theorem~2]{F1950}, or~\cite[Lemma~5.4]{GW2016b}. Thus the equivalence classes of the relation ``joinable by a path of finite image length'' are open. Connectedness of $M$ implies that there is only one such class. Symmetry and the triangle inequality follow from reversal and concatenation, and $\ell_{\rho_N}(f\circ\gamma)\ge\rho_N(f(x),f(x'))$ gives the lower bound. Let $x\ne x'$ with $f(x)=f(x')=y$. Choose a normal neighborhood $U$ of $x$ with $x'\notin\overline U$, and let $\tau=\dist_{\rho_N}(y,\partial f(U))>0$. Every path from $x$ to $x'$ meets $\partial U$, whose image is $\partial f(U)$; hence $d_f(x,x')\ge\tau$. For compatibility, let $x_k\to x$ in $M$. For large $k$, $x_k\in U$ and the minimizing geodesic $\beta_k$ from $f(x_k)$ to $f(x)$ lies in $f(U)$. Its lift $\alpha_k$ starting at $x_k$ (which exists by the path-lifting theorem for discrete open maps~\cite[Theorem~2]{F1950}\cite[Lemma~5.4]{GW2016b}) is connected, starts in $U$ and cannot meet $\partial U$, so it ends in $U\cap f^{-1}(f(x))=\{x\}$; thus $d_f(x_k,x)\le\rho_N(f(x_k),f(x))\to0$. Conversely, if $d_f(x_k,x)\to0$ but $x_k\notin U$ infinitely often, then $d_f(x_k,x)\ge\tau$ for those $k$, a contradiction.

(ii) For a partition of the parameter interval, $d_f(\gamma(t_{j-1}),\gamma(t_j))\le\ell_{\rho_N}(f\circ\gamma|_{[t_{j-1},t_j]})$, so $\ell_{d_f}(\gamma)\le\ell_{\rho_N}(f\circ\gamma)$; the reverse inequality follows from $d_f\ge\rho_N\circ(f\times f)$. Then $d_f(x,x')=\inf\ell_{d_f}(\gamma)$, so $d_f$ is a length metric, geodesic by compactness, and $\ell_{d_f}=\ell_{\rho_N}\circ f$ is the $1$-BLD condition.

(iii) Let $D$ be the length metric of Lemma~\ref{lem2.4}. For every path, $L^{-1}\ell_d(\gamma)\le\ell_{\rho_N}(f\circ\gamma)=\ell_{d_f}(\gamma)\le L\ell_d(\gamma)$; take infima over paths joining $x,x'$ and use $d\le D\le Cd$.
\end{proof}

We next record the two consequences of path lifting that will be used repeatedly.

\begin{lemma}\label{lem4.2}
Let $(M,D)$ be a compatible length metric and $f\colon(M,D)\to(N,\rho_N)$ $L$-BLD. Then $B_{\rho_N}(f(x),t/L)\subset f(B_D(x,t))$ for $x\in M$ and $0<t<r_N$. Moreover, if $C$ is a component of $f^{-1}(B(y_0,s))$ and $x\in C$, then the lift starting at $x$ of any path in $B(y_0,s)$ starting at $f(x)$ stays in $C$.
\end{lemma}
\begin{proof}
For $y'\in B(f(x),t/L)$ let $\beta$ be the minimizing geodesic from $f(x)$ to $y'$; its lift from $x$ has length less than $t$ by Lemma~\ref{lem2.5}. The second assertion holds because the lift is connected, starts in $C$ and lies in $f^{-1}(B(y_0,s))$.
\end{proof}

\begin{definition}\label{def4.3}
For $y\in N$ and $0<s<r_N$, we call a connected component of $f^{-1}(B(y,s))$ a \emph{cluster over $B(y,s)$}. This is terminology for the usual components of inverse images of normal domains; see~\cite[Chapter~I]{R1993}\cite[Proposition~4.2 and Section~5.3]{GW2016b}. Properness of $f$ and the normal-domain theory imply that each cluster $C$ is a normal domain with $f(C)=B(y,s)$ and $f(\partial C)=\partial B(y,s)$; we write $d_C$ for its degree. The degrees of the clusters over $B(y,s)$ sum to $m$, so there are at most $m$ clusters; a cluster of degree one is mapped homeomorphically onto $B(y,s)$ and contains no branch point; and if $B(y',s')\subset B(y,s)$, each cluster over $B(y',s')$ lies in exactly one cluster over $B(y,s)$, each cluster over $B(y,s)$ contains at least one cluster over $B(y',s')$, and the degrees of the sub-clusters inside a cluster $C$ sum to $d_C$.
\end{definition}

\begin{lemma}\label{lem4.4}
Let $f\colon(M,D)\to(N,\rho_N)$ be $L$-BLD, let $C$ be a cluster over $B(y,s)$, $0<s<r_N$, and $F=C\cap f^{-1}(y)$. Then the graph on $F$ joining $x,x'$ when $D(x,x')\le2Ls$ is connected. Consequently $\diam_DF\le2Ls(\#F-1)\le2Ls(d_C-1)$, every point of $C$ lies within $D$-distance $Ls$ of $F$, and $\diam_DC\le2Ls\,d_C$.
\end{lemma}
\begin{proof}
Each $z\in C$ is joined to a point of $F$ by the lift of the minimizing geodesic from $f(z)$ to $y$, which stays in $C$ and has length less than $Ls$ (Lemma~\ref{lem4.2}). If $F=F'\sqcup F''$ with $D(F',F'')>2Ls$, choose $\eta>0$ with $2(Ls+\eta)<D(F',F'')$; the open sets $\bigcup_{x\in F'}B(x,Ls+\eta)$ and $\bigcup_{x\in F''}B(x,Ls+\eta)$ are disjoint, both meet $C$, and their union contains $C$, contradicting connectedness. The remaining statements follow.
\end{proof}

\begin{theorem}\label{thm4.5}
For $0<r<r_N$ and $x\in M$,
\[
 \vol_N\bigl(B(f(x),r)\bigr)\le\HH^n_{d_f}\bigl(B_{d_f}(x,r)\bigr)\le C(n)\,m^{n+1}\,r^n .
\]
In particular $(M,d_f)$ is Ahlfors $n$-regular, with constants depending only on $n$, $m$ and $N$.
\end{theorem}
\begin{proof}
By Proposition~\ref{prop4.1}(ii), $f$ is $1$-Lipschitz from $(M,d_f)$, so it does not increase $\HH^n$, and by Lemma~\ref{lem4.2} (with $L=1$) $f(B_{d_f}(x,r))\supset B(f(x),r)$; this gives the lower bound. For the upper bound, cover $B(f(x),r)$ by at most $C(n)(r/t)^n$ balls $B(a_i,t)$, $0<t\le r$. The preimage of each is the union of at most $m$ clusters, each of $d_f$-diameter at most $2mt$ by Lemma~\ref{lem4.4}. Hence $f^{-1}(B(f(x),r))\supset B_{d_f}(x,r)$ is covered by at most $C(n)m(r/t)^n$ sets of diameter at most $2mt$, so the Hausdorff pre-measure at level $2mt$ of $B_{d_f}(x,r)$ is at most $C(n)m(r/t)^n(2mt)^n=2^nC(n)m^{n+1}r^n$. Let $t\to0$. For completeness, the large-scale constants can also be chosen independently of $f$. Put $D_N=\diam N$. Every point of $M$ can be joined to the fiber over a fixed $y\in N$ by a lifted target geodesic of length at most $D_N$. The same connectedness argument as in Lemma~\ref{lem4.4}, using open neighborhoods of radius $D_N+\varepsilon$ and then letting $\varepsilon\downarrow0$, gives $\diam_{d_f}M\le2mD_N$. A cover of $N$ by a fixed finite number of balls of radius $r_N/2$ and the preceding covering argument bound $\HH^n_{d_f}(M)$ in terms of $n,m,N$. For $r\ge r_N$, the lower bound at $r_N/2$, together with $r\le2mD_N$, and this total-measure bound give the two Ahlfors estimates.
\end{proof}

\begin{proof}[Proof of Theorem~\textup{\ref{thm1.4}}]
(i) is Proposition~\ref{prop4.1}(i),(ii) and Theorem~\ref{thm4.5}; (ii) is Proposition~\ref{prop4.1}(iii). For (iii): if $(M,d_f)$ is linearly locally contractible, $d_f$ itself is admissible in both senses, being geodesic (hence $1$-quasiconvex), Ahlfors $n$-regular, and $1$-BLD. Conversely, given an admissible metric $d$, quasiconvex by hypothesis or by Corollary~\ref{cor2.3}, Proposition~\ref{prop4.1}(iii) shows that $d$ is bi-Lipschitz to $d_f$, and linear local contractibility is bi-Lipschitz invariant: $B_{d_f}(x,r)\subset B_d(x,Lr)$ contracts in $B_d(x,\lambda Lr)\subset B_{d_f}(x,\lambda L^2Cr)$.
\end{proof}

\begin{remark}\label{rem4.6}
Theorem~\ref{thm4.5} also shows that the branch set may carry positive $n$-dimensional measure in the pullback metric: since $f$ is $1$-Lipschitz, $\HH^n_{d_f}(B_f)\ge\HH^n_{\rho_N}(f(B_f))$, which is positive for the examples of Theorem~\ref{thm1.7}(i). Ahlfors regularity is nevertheless automatic.
\end{remark}

\section{Clusters, skeletal extension, and the cluster obstruction}\label{sec5}

From now on $(M,D)$ is a closed connected oriented topological $n$-manifold with a compatible length metric and contractibility data $(\lambda,r_*)$, $(N,\rho_N)$ is a closed connected oriented Riemannian $n$-manifold, and $f\colon(M,D)\to(N,\rho_N)$ is an orientation-preserving $L$-BLD branched cover of degree $m\ge1$.

\subsection{Skeletal extension}
The following is the finite-complex version of the controlled extension principle in~\cite[Lemma~4.5]{GW2017}. We record a proof to specify the constants.
\begin{lemma}\label{lem5.1}
Define $C_0=1$ and $C_k=\lambda(C_{k-1}+2)+1$ for $k\ge1$. Let $a>0$ satisfy $(C_n+2)a<r_*$. Let $K$ be a finite simplicial complex of dimension at most $n$ with points $x_v\in M$ assigned to its vertices, and suppose that $D(x_v,x_w)\le a$ whenever $v,w$ span a simplex. Then the assignment extends to a continuous $g\colon|K|\to M$ with $g(\sigma)\subset B_D(x_v,C_ka)$ for every $k$-simplex $\sigma$ and vertex $v$ of $\sigma$.
\end{lemma}
\begin{proof}
Suppose $g$ is defined on the $(k-1)$-skeleton with the stated control, and let $\sigma$ be a $k$-simplex with a vertex $v_0$. For $z\in\partial\sigma$ in a face with vertex $w$, $D(g(z),x_{v_0})\le D(g(z),x_w)+D(x_w,x_{v_0})<(C_{k-1}+1)a$, so $g(\partial\sigma)\subset B(x_{v_0},(C_{k-1}+2)a)$, a ball of radius less than $r_*$ that contracts in $B(x_{v_0},\lambda(C_{k-1}+2)a)$. Composing the contraction with $g|_{\partial\sigma}$ gives an extension over the cone $\sigma$ with image in that ball, hence within $\lambda(C_{k-1}+2)a+a=C_ka$ of every vertex of $\sigma$. The extensions agree on common faces; induct.
\end{proof}
Only null-homotopies of spheres $\mathbb S^{k-1}$, $k\le n$, in balls of $\lambda$ times the radius are used, so linear local $(n-1)$-connectedness would suffice.

\subsection{Target coordinates}
\begin{lemma}\label{lem5.2}
There is $r_N>0$ (smaller than the one fixed above if necessary) such that for every $y_0\in N$ an oriented orthonormal identification $T_{y_0}N\cong\R^n$ gives a map $p_{y_0}=\exp_{y_0}$ which, for $0<r<r_N$, is an orientation-preserving embedding of the cube $Q_r=[-r,r]^n$ with
\begin{equation}\label{eq5.1}
 \tfrac12|z-w|\le\rho_N(p_{y_0}(z),p_{y_0}(w))\le2|z-w|,\qquad \rho_N(p_{y_0}(z),y_0)=|z|,
\end{equation}
so that $B(y_0,r)\subset p_{y_0}(Q_r)\subset\overline B(y_0,\sqrt n\,r)$, and points at distance less than $r_N$ are joined by a unique minimizing geodesic depending continuously on its endpoints.
\end{lemma}
\begin{proof}
Positivity of the injectivity radius, the Gauss lemma and compactness~\cite[Chapters~5--6]{Lee2018}.
\end{proof}

We fix the constants
\begin{equation}\label{eq5.2}
 \begin{aligned}
 A_0&=LC_n(2L+1)+2,\qquad\varepsilon_0=\frac1{8A_0},\\
 s(r)&=(2\sqrt n+L)r,\\
 r_2&=\min\Bigl\{1,\frac{r_N}{2\sqrt n+L},\frac{r_*}{(C_n+2)(2L+1)}\Bigr\}.
 \end{aligned}
\end{equation}
Note that $\varepsilon_0<1/(16L^2)$.

\subsection{A homological lemma}
\begin{lemma}\label{lem5.3}
Let $C\subset M$ be open and connected and $K\subset C$ a nonempty compact connected set. Then $H_n(C,C\setminus K)\cong\Z$, generated by the image $o_K$ of $[M]$ under $H_n(M)\to H_n(M,M\setminus K)\cong H_n(C,C\setminus K)$, and for every finite $F\subset K$ the image of $o_K$ in $H_n(C,C\setminus F)\cong\bigoplus_{x\in F}\Z o_x$ is $\sum_{x\in F}o_x$.
\end{lemma}
\begin{proof}
Excision applies because $K$ is compact and contained in the open set $C$: the closed set $M\setminus C$ lies in the interior of $M\setminus K$. Hence it is enough to compute $H_n(M,M\setminus K)$.

For a class $\alpha$ in this group write its localization at $x\in K$ as $a(x)o_x$, with $a(x)\in\Z$. The integer-valued function $a$ is locally constant on $K$. To see this directly, represent $\alpha$ by a finite relative cycle $c$, so the compact support of $\partial c$ is disjoint from $K$. Around any $x\in K$ choose a closed coordinate ball $D$ whose interior contains $x$ and which misses that support. The same chain defines a class in $H_n(M,M\setminus D)\cong\Z$, and its localization has the same coefficient at every point of $\operatorname{int}D$. This proves local constancy. Since $K$ is connected, $a(x)$ is one integer $q$ throughout $K$; no local connectedness or path connectedness of $K$ is needed.

Let $o_K$ be the image of the global orientation class $[M]$. By~\cite[Theorem~3.26]{Ha2002}, its localization is $o_x$ at each $x\in K$. Consequently $\alpha-qo_K$ localizes to zero everywhere on $K$. The uniqueness assertion for compact sets in~\cite[Lemma~3.27(a)]{Ha2002} implies $\alpha-qo_K=0$. Also $o_K$ has infinite order, since its localization at any point is a generator of $\Z$. Thus $H_n(C,C\setminus K)=\Z o_K$.

Finally, disjoint coordinate neighborhoods of the finitely many points of $F$, followed by excision, identify $H_n(C,C\setminus F)$ with $\bigoplus_{x\in F}\Z o_x$. Naturality of localization sends $o_K$ to the element with coefficient $1$ at every point, namely $\sum_{x\in F}o_x$.
\end{proof}

\subsection{The cluster obstruction}
\begin{proposition}\label{prop5.4}
Let $y_0\in N$, $0<r<r_2$, $s=s(r)$, and let $C$ be a cluster over $B(y_0,s)$ with $d_C\ge2$. Then there is $y\in p_{y_0}(Q_r)$ with
\begin{equation}\label{eq5.3}
 \diam_D\bigl(C\cap f^{-1}(y)\bigr)>\varepsilon_0r.
\end{equation}
\end{proposition}
\begin{proof}
Write $p=p_{y_0}$ and let
\[
 K=C\cap f^{-1}(\overline B(y_0,r/4))
   =\overline C\cap f^{-1}(\overline B(y_0,r/4)).
\]
The equality follows from $f(\partial C)=\partial B(y_0,s)$ and $r/4<s$. In particular, $K\subset C$ is compact, being the intersection of two closed subsets of the compact manifold $M$. The components $W_1,\dots,W_k$ of $C\cap f^{-1}(B(y_0,r/4))$ are normal domains mapped onto $B(y_0,r/4)$, $k\le d_C$, and by openness of $f$ every point of $K$ lies in the closure of their union; thus $K=\overline{W_1}\cup\dots\cup\overline{W_k}$ has finitely many components, each containing a point of $C\cap f^{-1}(y_0)$.

\emph{Case 1: $K$ is disconnected.} Choose $x,x'\in C\cap f^{-1}(y_0)$ in different components of $K$ and let $\gamma$ be a path from $x$ to $x'$. If $\gamma$ leaves $C$ it meets $\partial C$, so $f\circ\gamma$ reaches distance $s>r/2$ from $y_0$; if $\gamma$ stays in $C$ it is not contained in $K$, so $f\circ\gamma$ leaves $\overline B(y_0,r/4)$ and returns to $y_0$. In both cases $\ell_{\rho_N}(f\circ\gamma)\ge r/2$, hence $\ell_D(\gamma)\ge r/(2L)$ by~\eqref{eq2.3}, and $D(x,x')\ge r/(2L)>\varepsilon_0r$. Take $y=y_0$.

\emph{Case 2: $K$ is connected.} Suppose that $\diam_D(C\cap f^{-1}(y))\le\varepsilon_0r$ for all $y\in p(Q_r)$. Triangulate $Q_r$ with mesh at most $\varepsilon_0r$ and choose $x_v\in C\cap f^{-1}(p(v))$ for each vertex $v$; this is possible since $f(C)=B(y_0,s)\supset p(Q_r)$. If $v,w$ span a simplex, the lift from $x_v$ of the minimizing geodesic from $p(v)$ to $p(w)$ ends at some $x'_w\in C\cap f^{-1}(p(w))$ (Lemma~\ref{lem4.2}) at $D$-distance at most $L\rho_N(p(v),p(w))\le2L\varepsilon_0r$ from $x_v$, so $D(x_v,x_w)\le(2L+1)\varepsilon_0r=:a$. By~\eqref{eq5.2}, $(C_n+2)a<r_*$, and Lemma~\ref{lem5.1} gives $g\colon Q_r\to M$ with $g(\sigma)\subset B_D(x_v,C_na)$ for every simplex $\sigma$ and vertex $v$ of $\sigma$. Any point within $D$-distance $C_na$ of $x_v$ is joined to $x_v$ by a path whose image has length less than $LC_n(2L+1)\varepsilon_0r<r/8$ and therefore stays in $B(y_0,\sqrt n\,r+r/8)\subset B(y_0,s)$; hence $g(Q_r)\subset C$. As in the estimate just made,
\begin{equation}\label{eq5.4}
 \rho_N\bigl(f(g(z)),p(z)\bigr)\le L\,D(g(z),x_v)+2\varepsilon_0r\le A_0\varepsilon_0r=r/8\qquad(z\in Q_r).
\end{equation}

Let $V=B(y_0,s)$ and $\overline B=\overline B(y_0,r/4)$. The minimizing geodesics from $p(z)$ to $f(g(z))$ give a homotopy from $p$ to $f\circ g$ whose points lie in $V$ and, for $z\in\partial Q_r$, at distance at least $r-r/8>r/4$ from $y_0$. Hence $p$ and $f\circ g$ are homotopic as maps of pairs $(Q_r,\partial Q_r)\to(V,V\setminus\overline B)$, and $p_*[Q_r,\partial Q_r]=o_{\overline B}$, the generator of $H_n(V,V\setminus\overline B)\cong\Z$ (Lemma~\ref{lem5.3} in $N$), because $p$ is an orientation-preserving embedding of $Q_r\supset p^{-1}(\overline B)$ and both classes localize to $o_{y_0}$. On the source side, $g(\partial Q_r)\cap K=\emptyset$ by~\eqref{eq5.4}, and $f(C\setminus K)\subset V\setminus\overline B$ by the definition of $K$, so
\[
 (Q_r,\partial Q_r)\xrightarrow{\ g\ }(C,C\setminus K)\xrightarrow{\ f\ }(V,V\setminus\overline B)
\]
are maps of pairs. Write $g_*[Q_r,\partial Q_r]=q\,o_K$. Fix $z\in B(y_0,r/4)$ and $F=C\cap f^{-1}(z)\subset K$. Naturality of the inclusions $(C,C\setminus K)\to(C,C\setminus F)$ and $(V,V\setminus\overline B)\to(V,V\setminus\{z\})$ together with Lemma~\ref{lem5.3} gives
\[
 f_*o_K=\Bigl(\sum_{x\in F}i(x,f)\Bigr)o_{\overline B}=d_C\,o_{\overline B}.
\]
Therefore $(f\circ g)_*[Q_r,\partial Q_r]=d_Cq\,o_{\overline B}=o_{\overline B}$, i.e.\ $d_Cq=1$, impossible for $d_C\ge2$.
\end{proof}

\begin{remark}\label{rem5.5}
When $C\cap f^{-1}(y_0)$ is a single point $x_0$, one can instead work with the pairs punctured at $x_0$ and $y_0$, and the argument reduces to the classical computation $1=(\deg)\cdot q$ at a totally ramified value. When the fiber over $y_0$ has several points with local indices $d_1,\dots,d_k$, the pointwise argument only gives $1=\sum d_jq_j$, which is no obstruction if $\gcd(d_j)=1$; replacing the fiber by the compact connected set $K$ replaces the greatest common divisor by the total degree $d_C$, which is what makes Theorem~\ref{thm1.5} possible in every degree.
\end{remark}

\section{Uniform porosity and the Newman property}\label{sec6}

\begin{theorem}\label{thm6.1}
In the setting of Section~\textup{\ref{sec5}} with $m\ge2$, put
\[
 \theta=\frac{\varepsilon_0}{8L(m+1)(2\sqrt n+L)},\qquad\delta=\tfrac12\theta^{m-1},\qquad R_2=(2\sqrt n+L)r_2 .
\]
Then for every $y_0\in f(B_f)$ and $0<R<R_2$ there is a ball $B(y,\delta R)\subset B(y_0,R)\setminus f(B_f)$.
\end{theorem}
\begin{proof}
Let $s_0=R/2$, $r_0=s_0/(2\sqrt n+L)<r_2$. We construct $y_i$ and $r_i$ with $s_i=(2\sqrt n+L)r_i$, $r_{i+1}=\theta r_i$, $B(y_{i+1},s_{i+1})\subset B(y_i,s_i)$, such that the number $k_i$ of clusters over $B(y_i,s_i)$ satisfies $k_{i+1}\ge k_i+1$. Since $k_i\le m$, the construction stops after at most $m-1$ steps.

Given $(y_i,r_i)$: if all clusters over $B(y_i,s_i)$ have degree one, then $B(y_i,s_i)\cap f(B_f)=\emptyset$ and we stop. Otherwise pick a cluster $C$ with $d_C\ge2$ and apply Proposition~\ref{prop5.4} with center $y_i$ and radius $r_i$: there are $y_{i+1}\in p_{y_i}(Q_{r_i})\subset\overline B(y_i,\sqrt n\,r_i)$ and $x,x'\in C\cap f^{-1}(y_{i+1})$ with $D(x,x')>\varepsilon_0r_i$. Let $r_{i+1}=\theta r_i$, so $s_{i+1}=\varepsilon_0r_i/(8L(m+1))$ and $\sqrt n\,r_i+3s_{i+1}<s_i$. The clusters over $B(y_{i+1},s_{i+1})$ contained in $C$ are components of $C\cap f^{-1}(B(y_{i+1},s_{i+1}))$; by Lemma~\ref{lem4.4} applied to each such cluster (with center $y_{i+1}$ and radius $s_{i+1}<r_N$), each has diameter at most $2Ls_{i+1}m<\varepsilon_0r_i$. Hence $x$ and $x'$ lie in different clusters over $B(y_{i+1},s_{i+1})$. By the nesting properties in Definition~\ref{def4.3}, $C$ contains at least two of the new clusters and every other old cluster contains at least one, so $k_{i+1}\ge k_i+1$.

At the final step $j\le m-1$, $B(y_j,s_j)\cap f(B_f)=\emptyset$, $B(y_j,s_j)\subset B(y_0,s_0)\subset B(y_0,R)$, and $s_j=\theta^js_0\ge\theta^{m-1}R/2=\delta R$.
\end{proof}

\begin{corollary}\label{cor6.2}
Under the assumptions of Theorem~\ref{thm6.1}, $f(B_f)$ is uniformly $\delta$-porous at scales below $R_2$, $\dim_{\mathcal H}f(B_f)\le\dim_Af(B_f)\le n-c(n,\delta)<n$, and $\vol_N(f(B_f))=0$.
\end{corollary}
\begin{proof}
$N$ is locally $2$-bi-Lipschitz to $\R^n$ at scales below $r_N$ (Lemma~\ref{lem5.2}), porosity is preserved by bi-Lipschitz maps up to constants, and a uniformly porous subset of $\R^n$ has Assouad dimension at most $n-c(n,\delta)$~\cite{Luu1998}; Hausdorff dimension never exceeds Assouad dimension, and a set of dimension less than $n$ is null.
\end{proof}

\begin{corollary}\label{cor6.3}
Let $f\colon M\to N$ be an orientation-preserving branched cover of degree $m\ge2$ between a closed connected oriented topological $n$-manifold and a closed connected oriented Riemannian $n$-manifold. If $f$ admits a compatible quasiconvex, linearly locally contractible metric for which it is $L$-BLD, then $f(B_f)$ is uniformly porous at small scales with constant $\delta(n,L,C\lambda,m)$. In particular $\vol_N(f(B_f))>0$, or merely $\dim_{\mathcal H}f(B_f)=n$, excludes all such metrics, and, if $M$ is homeomorphic to a smooth manifold, all Ahlfors $n$-regular linearly locally contractible ones.
\end{corollary}
\begin{proof}
Lemma~\ref{lem2.4} replaces the metric by a length metric with the same path lengths and contractibility data $(C\lambda,r_*)$; apply Theorem~\ref{thm6.1} and Corollary~\ref{cor6.2}. For Ahlfors regular metrics use Corollary~\ref{cor2.3}.
\end{proof}

\begin{theorem}\label{thm6.4}
Under the assumptions of Theorem~\ref{thm6.1} there is $t_2>0$, depending on $n,L,\lambda,m,r_*,N$, such that for every $z\in B_f$ and $0<t<t_2$ there is $w\in B_D(z,t)$ with
\[
 \diam_D\bigl(B_D(z,t)\cap f^{-1}(f(w))\bigr)\ge\varepsilon_N\,t,\qquad\varepsilon_N=\varepsilon_0\theta .
\]
Consequently $B_f$ is uniformly porous in $(M,D)$ at small scales, and $B_f$ is a null set for every doubling measure on $(M,D)$.
\end{theorem}
\begin{proof}
Let $y_0=f(z)$, $r=\theta t$ and $s=(2\sqrt n+L)r$. Choose $t_2>0$ so that $r<r_2$, $2s<r_N$ and $t/(4L)<R_2$ whenever $t<t_2$. Let $C$ be the cluster over $B(y_0,s)$ containing $z$. Lemma~\ref{lem4.4} gives
\[
 \diam_DC\le2Lsm
 =2L(2\sqrt n+L)m\theta t
 \le\varepsilon_0t/4<t.
\]
Since $z\in C$, this implies $C\subset B_D(z,t)$. Also $d_C\ge i(z,f)\ge2$. Proposition~\ref{prop5.4} gives a value $y$ and two points $x,x'\in C\cap f^{-1}(y)$ with
\[
 D(x,x')>\varepsilon_0r=\varepsilon_0\theta t.
\]
Taking $w=x$ proves the stated estimate. Applying this uniform estimate to branch-centered balls in a sufficiently small neighborhood of any fixed $z_0\in B_f$ gives the Newman property at every branch point in the sense of~\cite{L2019}. The ball centers in that definition lie in $B_f$, as recalled in Section~\ref{sec2}.

We give a direct quantitative deduction of source porosity. For $z\in B_f$ and $0<t<t_2$, apply Theorem~\ref{thm6.1} with $R=t/(4L)$ to obtain
\[
 B(y,\delta R)\subset B(f(z),R)\setminus f(B_f).
\]
Lift a minimizing geodesic from $f(z)$ to $y$ starting at $z$, and denote its endpoint by $w$. Lemma~\ref{lem2.5} gives $D(z,w)<LR=t/4$. Set $\delta_B=\delta/(8L^2)$. Since $f$ is $L$-Lipschitz,
\[
 f(B_D(w,\delta_Bt))\subset B(y,\delta R/2),\qquad
 B_D(w,\delta_Bt)\subset B_D(z,t)\setminus B_f.
\]
This proves uniform porosity with an explicit constant and is consistent with Luisto's Newman--porosity implication~\cite[Theorem~1.1]{L2019}. A uniformly porous closed set is null for any locally finite doubling Borel measure with full support: a hole of a fixed relative radius has a fixed positive fraction of the measure of its containing ball, contradicting the density theorem at almost every point of a positive-measure set; see~\cite{H2001,M1995}. This is the convention for doubling measures in the statement.
\end{proof}

\begin{proof}[Proof of Theorem~\textup{\ref{thm1.5}}]
For part~\textup{(i)}, Theorem~\ref{thm6.1} supplies a hole in the entire branch-value set at every radius below $R_2$, with $\delta=\tfrac12\theta^{m-1}$ depending only on $n,L,\lambda,m$. Corollary~\ref{cor6.2} converts this uniform porosity into the asserted Assouad and Hausdorff dimension bounds and the zero-volume conclusion. For part~\textup{(ii)}, Theorem~\ref{thm6.4} proves the Newman estimate with $\varepsilon_N=\varepsilon_0\theta$ and source porosity with $\delta_B=\delta/(8L^2)$. Equations~\eqref{eq5.2} and the choice of $t_2$ show the stated dependence of the scale thresholds. The cluster obstruction and these deductions use no source measure.
\end{proof}

The following lemma explains what the recursion in Theorem~\ref{thm6.1} is looking for. Write $\sep(y)=\min\{D(x,x')\colon x\ne x'\in f^{-1}(y)\}$ if $\#f^{-1}(y)=m$ and $\sep(y)=0$ otherwise.

\begin{lemma}\label{lem6.5}
Let $y\in N$ and $0<s<r_N$. If $\sep(y)>2Ls$ then $B(y,s)\cap f(B_f)=\emptyset$; if $B(y,s)\cap f(B_f)=\emptyset$ then $\sep(y)\ge2s/L$.
\end{lemma}
\begin{proof}
If $\sep(y)>2Ls$, Lemma~\ref{lem4.4} shows that each cluster over $B(y,s)$ contains one fiber point, so there are $m$ clusters, all of degree one, and none contains a branch point. Conversely, if $B(y,s)$ contains no branch value, $f$ is a covering over the simply connected ball, the $m$ preimages of $y$ lie in different clusters, and a path joining two of them has an image that leaves $B(y,s)$ and returns, of length at least $2s$; use~\eqref{eq2.3}.
\end{proof}

\subsection{Porous branch values do not suffice}
We now verify the geometry of the branch-value sets in Section~\ref{sec3}. This makes the connection between the essential-loop obstruction and the porosity theorem explicit.

\begin{proposition}\label{prop6.6}
For the covers $F_n$ constructed in Section~\ref{sec3}, the sets $A_n=F_n(B_{F_n})$ satisfy \textup{(i)}--\textup{(ii)} of Corollary~\ref{cor1.6}.
\end{proposition}
\begin{proof}
In dimension three, $A_3=\psi(\Gamma)=h(E)$. The construction of $T$ and $S$ gives the full formula
\[
 \Gamma=
 \{(s,\pm s^2,0):0\le s\le1\}
 \,\cup\,
 \{(s,\pm(2-s)^2,0):1\le s\le2\}.
\]
Thus $\Gamma$ is a finite union of smooth rectifiable arcs with at least one nondegenerate smooth subarc. The stereographic chart and its inverse are bi-Lipschitz on a compact neighborhood of $\Gamma$, so $\dim_{\mathcal H}A_3=1$. Moreover, $\Gamma\subset\{z=0\}$ and the chosen coordinates identify that plane with a subset of the great sphere $G_3$. Finally, the ambient homeomorphism $h$ carries the standard circle $E$ onto $A_3$.

Suppose $n>3$ and put $k=n-3$. The product description in Lemma~\ref{lem3.5} gives
\begin{equation}\label{eq6.1}
 A_n=
 \bigl\{(\sqrt{1-|v|^2}\,a,v):a\in A_3,\ |v|<1\bigr\}
 \,\cup\,\bigl(\{0\}\times\mathbb S^{k-1}\bigr).
\end{equation}
Indeed, in the product region branching occurs exactly over $A_3$. Every point of the remaining sphere is also a branch point: every one of its neighborhoods contains two points with the same image, by the degree-two product description, so $F_n$ cannot be locally injective there. These points are fixed by $F_n$.

Let $V\subset\R^4$ be the three-dimensional linear subspace with $G_3=V\cap\mathbb S^3$. Equation~\eqref{eq6.1} places $A_n$ in
\[
 G_n=(V\times\R^k)\cap\Sn,
\]
a round great $(n-1)$-sphere. Suspending $h$ defines an ambient homeomorphism
\[
 h_n(\sqrt{1-|v|^2}\,u,v)
   =(\sqrt{1-|v|^2}\,h(u),v),\qquad
 h_n(0,v)=(0,v).
\]
Continuity at $|v|=1$ follows as in Lemma~\ref{lem3.5}, and its inverse is obtained by suspending $h^{-1}$. It carries the standard equator $\mathbb S^{n-2}$, the corresponding suspension of $E$, onto $A_n$. This proves the ambient topological assertion.

For the dimension bound, parameterize $A_3$ by its four smooth arcs. On the product of each parameter interval with each closed ball $|v|\le1-1/j$, $j\ge2$, the map in~\eqref{eq6.1} is Lipschitz. These countably many images cover the part with $|v|<1$ and have Hausdorff dimension at most $1+k=n-2$. The remaining sphere has dimension $k-1=n-4$. Conversely, a nondegenerate smooth subarc of $A_3$ times a small open ball in the $v$-variables gives a smooth embedded submanifold of dimension $1+k$. Hence $\dim_{\mathcal H}A_n=n-2$.

It remains to quantify porosity, uniformly in the dimension. For every $n\ge3$, write $G_n=\Sn\cap\nu^\perp$ with $|\nu|=1$. Given $a\in A_n$ and $0<R<\pi$, set
\[
 y=\cos(R/2)\,a+\sin(R/2)\,\nu.
\]
Since $R/2<\pi/2$, we have $\rho_n(a,y)=R/2$ and $\dist_{\rho_n}(y,G_n)=R/2$. The triangle inequality therefore gives
\[
 B_{\rho_n}(y,R/4)\subset
 B_{\rho_n}(a,R)\setminus G_n
 \subset B_{\rho_n}(a,R)\setminus A_n.
\]
This is the asserted uniform $1/4$-porosity.
\end{proof}

\begin{proof}[Proof of Corollary~\ref{cor1.6}]
Theorem~\ref{thm1.2} provides the degree-two covers and excludes every compatible quasiconvex linearly locally contractible BLD source metric for them. Proposition~\ref{prop6.6} proves \textup{(i)}--\textup{(ii)}. By Theorem~\ref{thm1.4}(i), $d_{F_n}$ is compatible, geodesic and Ahlfors $n$-regular, and $F_n$ is $1$-BLD from this metric. If it were linearly locally contractible, it would itself be a quasiconvex source metric forbidden by Theorem~\ref{thm1.2}. This proves \textup{(iii)} and the failure of sufficiency.
\end{proof}

\section{Target changes, measure obstructions, and the planar case}\label{sec7}

The previous section gives nonadmissible covers with small, uniformly porous branch-value sets. We now compare them with the standard positive examples, the positive-volume obstruction, and target changes that preserve admissibility.

\subsection{The winding cover}
Write $\Sn=\{(z,u)\in\C\times\R^{n-1}\colon|z|^2+|u|^2=1\}$ and $E=\{(0,u)\colon|u|=1\}$. For $m\ge2$ let $q_m(z,u)=(z^m/|z|^{m-1},u)$ for $z\ne0$ and $q_m(0,u)=(0,u)$.

\begin{lemma}\label{lem7.1}
$q_m\colon\Sn\to\Sn$ is an orientation-preserving branched cover of degree $m$ with $B_{q_m}=q_m(B_{q_m})=E$, $q_m^{-1}(a)=\{a\}$ for $a\in E$, and exactly $m$ preimages of every point outside $E$. It is $1$-BLD from $(\Sn,d_{q_m})$, and $(\Sn,d_{q_m})$ is linearly locally contractible; in particular $q_m$ admits an admissible metric.
\end{lemma}
\begin{proof}
The map preserves $|z|$ and $u$. Away from $E$ it is, in polar coordinates of the $z$-plane, $(r,\vartheta,u)\mapsto(r,m\vartheta,u)$, a local orientation-preserving homeomorphism with $m$ preimages of each point. Near $(0,u_0)\in E$ a tubular chart $(z,v)\mapsto(z,\sqrt{1-|z|^2}\varphi(v))$ turns it into the product of the planar winding map $re^{i\vartheta}\mapsto re^{im\vartheta}$ with the identity, which is open, has local degree $m$ on $\{z=0\}$, and is not locally injective there. This proves the topological assertions. The planar winding map $re^{i\vartheta}\mapsto re^{im\vartheta}$ is the standard BLD winding map~\cite[Chapter~I, Example~3.1]{R1993}. Its product with the identity has the same length inequalities, as follows either from the product length formula or by applying that formula after arclength parametrization; infinite-length paths follow by the same inequalities on subpaths. The tubular charts just used are locally bi-Lipschitz to Euclidean product charts. Away from $E$ the differential of $q_m$ has singular values between $1$ and $m$ in the round metric. A finite cover by these charts therefore gives local BLD inequalities with one finite constant. Subdividing any path into these chart neighborhoods gives the global BLD inequalities, also when its length is infinite. The round source metric is consequently admissible, and Theorem~\ref{thm1.4}(iii) shows that $(\Sn,d_{q_m})$ is linearly locally contractible.
\end{proof}

\subsection{Transporting mass onto a tame sphere}
A Borel probability measure on a closed manifold is \emph{good} if it has no atoms and gives positive mass to every nonempty open set.

\begin{theorem}\cite{AP2000,Fa1980,OU1941}\label{thm7.2}
On a closed connected triangulable manifold of dimension at least two, any two good probability measures are carried to one another by a homeomorphism.
\end{theorem}

\begin{proposition}\label{prop7.3}
For $n\ge3$ and $0<\theta<1$ there is an orientation-preserving homeomorphism $h\colon\Sn\to\Sn$ with $v_n(h(E))=\theta$.
\end{proposition}
\begin{proof}
Let $\sigma_E$ be normalized round $(n-2)$-area on $E$, a nonatomic measure since $n-2\ge1$, and $\nu_\theta=(1-\theta)v_n+\theta\sigma_E$, a good measure. Theorem~\ref{thm7.2} gives $h$ with $h_*\nu_\theta=v_n$; composing with the reflection $(z,u)\mapsto(\bar z,u)$, which preserves $v_n$, $\sigma_E$ and $E$, we may assume $h$ orientation-preserving. Then $v_n(h(E))=\nu_\theta(E)=\theta$.
\end{proof}

The set $h(E)$ is ambiently homeomorphic to the standard equator, hence tame and locally flat; its positive volume is a metric, not a topological, property.

\subsection{Proof of Theorem~\ref{thm1.7}(i)}
\begin{proof}
Let $f=f_{m,\theta}=h\circ q_m$ and $A=h(E)$. Post-composition with a homeomorphism preserves openness, discreteness, degree and the branch set, so $\deg f=m$, $B_f=E$ and $f(B_f)=A$ with $v_n(A)=\theta>0$. If $f$ admitted a compatible quasiconvex linearly locally contractible metric for which it is BLD, Corollary~\ref{cor6.3} would give $v_n(A)=0$. Ahlfors regular admissible metrics are excluded through Corollary~\ref{cor2.3}. By Theorem~\ref{thm1.4}, $d_f$ is a compatible Ahlfors $n$-regular geodesic metric, $f$ is $1$-BLD from it, and $(\Sn,d_f)$ is not linearly locally contractible.
\end{proof}

\begin{corollary}\label{cor7.4}
For any metric $\rho'$ on the target bi-Lipschitz to $\rho$, $f_{m,\theta}$ admits no compatible quasiconvex linearly locally contractible source metric for which it is BLD onto $(\Sn,\rho')$.
\end{corollary}
\begin{proof}
Choose $K\ge1$ with $K^{-1}\rho\le\rho'\le K\rho$. An $L$-BLD realization with target $\rho'$ is then a $KL$-BLD realization with target $\rho$.
\end{proof}

\subsection{Quasisymmetric invariance}
\begin{theorem}\label{thm7.5}
Let $q_0\colon\Sn\to\Sn$ be an orientation-preserving branched cover of degree $m$ such that $(\Sn,d_{q_0})$ has contractibility data $(\lambda_0,r_0)$, and let $\Phi\colon\Sn\to\Sn$ be an orientation-preserving $\eta$-quasisymmetric homeomorphism. Then $(\Sn,d_{\Phi\circ q_0})$ is linearly locally contractible with data depending only on $n$, $\eta$, $m$, $\lambda_0$ and $r_0$. In particular, $\Phi\circ q_0$ admits an admissible metric.
\end{theorem}
\begin{proof}
Put $D_0=d_{q_0}$, $D_1=d_{\Phi\circ q_0}$, and $f_1=\Phi\circ q_0$. Both maps are $1$-BLD from the respective metrics. Fix $x\in\Sn$, put $a=f_1(x)$ and $b=q_0(x)$, and write $\Psi=\Phi^{-1}$. Its quasisymmetry function $\eta'$ depends only on $\eta$~\cite{H2001}. Define
\[
 u(r)=\max_{\rho(a,y)\le r}\rho(b,\Psi(y)),\qquad
 v(r)=\min_{\rho(a,y)\ge r}\rho(b,\Psi(y)).
\]
For small positive $r$, both are positive, and
\[
 B(b,v(r))\subset\Psi(B(a,r))\subset B(b,2u(r)).
\]
For $\kappa\ge1$ with $\kappa r<\pi$, quasisymmetry gives
\[
 u(r)\le\eta'(1/\kappa)v(\kappa r),\qquad
 u(r)\le\pi\eta'(r/\pi).
\]
The second estimate follows by comparing each $y\in\overline B(a,r)$ with the antipode of $a$. Choose $\kappa$ so that $5m\lambda_0\eta'(1/\kappa)\le1$. Choose $r_1>0$, depending only on the stated data, so that for $r<r_1$,
\[
 2u(r)<r_N,\qquad 5mu(r)<r_0,\qquad \kappa r<r_N.
\]
The preceding bound on $u(r)$ makes this choice uniform in $x$ and in $\Phi$ with the given quasisymmetry function.

Every point of $B_{D_1}(x,r)$ can be joined to $x$ by a path whose $f_1$-image has length less than $r$. Hence this source ball lies in the component of $q_0^{-1}(\Psi(B(a,r)))$ containing $x$, and thus in the component $C$ of $q_0^{-1}(B(b,2u(r)))$ containing $x$. Lemma~\ref{lem4.4} gives $\diam_{D_0}C\le4mu(r)$. Since $x\in C$, we have
\[
 C\subset\overline B_{D_0}(x,4mu(r))\subset B_{D_0}(x,5mu(r)).
\]
Contract this last open ball inside $B_{D_0}(x,5m\lambda_0u(r))$.

This larger $D_0$-ball is path connected, contains $x$, and its $q_0$-image lies in
\[
 B(b,5m\lambda_0u(r))\subset B(b,v(\kappa r))
 \subset\Psi(B(a,\kappa r)).
\]
The whole contraction therefore remains in the component of $f_1^{-1}(B(a,\kappa r))$ containing $x$. A second application of Lemma~\ref{lem4.4} bounds its $D_1$-diameter by $2m\kappa r$ and therefore places it inside
\[
 \overline B_{D_1}(x,2m\kappa r)
 \subset B_{D_1}(x,3m\kappa r).
\]
We have proved contractibility data $(3m\kappa,r_1)$ for $D_1$. Theorem~\ref{thm1.4} now gives admissibility.
\end{proof}

\begin{corollary}\label{cor7.6}
For every orientation-preserving quasiconformal $\Phi\colon\Sn\to\Sn$ and $m\ge2$ the cover $\Phi\circ q_m$ admits an admissible metric, and its branch-value set is $\Phi(E)$. For $n=3$, every planar quasicircle $\Gamma\subset\R^2\subset\mathbb S^3$ arises this way, so admissible branching along snowflake curves of any Hausdorff dimension in $(1,2)$ occurs.
\end{corollary}
\begin{proof}
Quasiconformal self-maps of $\Sn$ are quasisymmetric~\cite{H2001,V1971}; apply Lemma~\ref{lem7.1} and Theorem~\ref{thm7.5}. A planar quasicircle is the image of a round circle under an orientation-preserving quasiconformal self-map of $\R^2$. By~\cite[Theorem~2]{A1964}, this map extends quasiconformally to the upper half-space in $\R^3$; reflection across the boundary plane gives a quasiconformal self-map of all of $\R^3$. Its one-point compactification, composed with an orientation-preserving round M\"obius map carrying $E$ to that circle, gives the required map of $\mathbb S^3$. The general extension theorem~\cite{TV1982} also includes the case $\R^2\to\R^3$. Finally, snowflake quasicircles of every Hausdorff dimension in $(1,2)$ exist~\cite{GV1973}.
\end{proof}

\begin{remark}\label{rem7.7}
The homeomorphism $h$ of Proposition~\ref{prop7.3} carries the null set $E$ onto a set of positive volume and is therefore not quasisymmetric. Theorem~\ref{thm7.5} shows that post-composition with a quasisymmetric map preserves admissibility, while Theorem~\ref{thm1.7}(i) shows that post-composition with a homeomorphism does not; this does not assert that every non-quasisymmetric target change destroys admissibility. Also, admissibility of $\Phi\circ q_m$ for quasiconformal $\Phi$ shows that no condition of local bi-Lipschitz flatness on the branch-value set is necessary: the necessary condition of Corollary~\ref{cor6.3} (uniform porosity) allows branch-value sets of dimension greater than the topological branching dimension, as the dimension-three examples demonstrate.
\end{remark}

\subsection{The planar case}
\begin{proposition}\label{prop7.8}
Every orientation-preserving branched cover $f\colon\mathbb S^2\to\mathbb S^2$ admits an admissible metric, namely $d_f$.
\end{proposition}
\begin{proof}
By Sto\"ilow's theorem~\cite{LP2020,St1928} there are a homeomorphism $\varphi\colon\mathbb S^2\to\widehat\C$ and a nonconstant rational map $R$ with $f=R\circ\varphi$. Since $\varphi$ is a homeomorphism, $(\mathbb S^2,d_f)$ is isometric to $(\widehat\C,d_R)$. Near a point $w$ with $i(w,R)=k$ there are holomorphic coordinates in which $R(\zeta)=\zeta^k$, and the round metric is bi-Lipschitz to $|d\xi|$ on a target disk. On a small disk around $w$, $d_R$ is therefore comparable to the length metric of $k|\zeta|^{k-1}|d\zeta|$, the Euclidean cone of angle $2\pi k$: paths leaving the disk have image length bounded below, and paths inside have image length comparable to their cone length. A Euclidean cone of angle at least $2\pi$ is $\mathrm{CAT}(0)$~\cite[Theorem~II.3.14]{BH1999}, so its balls are convex and contractible in themselves. Linear local contractibility is local and bi-Lipschitz invariant, so finitely many such charts and compactness give contractibility data for $d_R$. Ahlfors regularity is Theorem~\ref{thm4.5}, and $R$ is $1$-BLD from $d_R$ by Proposition~\ref{prop4.1}.
\end{proof}

\begin{proof}[Proof of Theorem~\textup{\ref{thm1.7}}]
(i) was proved above, (ii) is Theorem~\ref{thm7.5} with Corollary~\ref{cor7.6}, and (iii) is Proposition~\ref{prop7.8}.
\end{proof}

\section{The dependence on the degree is necessary}\label{sec8}

\subsection{Local indices are unbounded}
\begin{example}\label{ex8.1}
For $w_k(z)=z^k$ on $\widehat\C$, $(\widehat\C,d_{w_k})$ is the spherical cone surface with two cone points of angle $2\pi k$. It is locally $\mathrm{CAT}(1)$ by the cone criterion. There is no nonconstant closed local geodesic of length less than $2\pi$: one avoiding the two cone points projects to a closed round geodesic; a segment between the two cone points has length at least $\pi$; and a geodesic segment from one cone point back to itself with no cone point in its interior would project to a round geodesic loop of length at least $2\pi$. The globalization criterion~\cite[Chapters~II.3--II.4]{BH1999} therefore gives $\mathrm{CAT}(1)$. Balls of radius less than $\pi/2$ are convex and contract to their centers along their unique geodesics. Thus $(\widehat\C,d_{w_k})$ has contractibility data $(1,\pi/2)$ while $i(0,w_k)=k$, so no bound $i(x,f)\le C(n,L,\lambda)$ holds.
\end{example}

\subsection{Path-tree covers with prescribed branch values}
We use the Riemann sphere $\widehat\C$ with the round metric $\rho$, normalized to have injectivity radius $\pi$, with $0$ and $\infty$ antipodal, so that $|z|<1$ is the open southern hemisphere. For $a\ne0,\infty$ let $\mathcal R_a=\{ta\colon t\ge1\}\cup\{\infty\}$, a geodesic arc from $a$ to $\infty$.

\begin{construction}\label{con8.2}
Let $A=\{a_1,\dots,a_d\}\subset\{0<|z|<1\}$ have pairwise distinct arguments, ordered so that $\arg a_1<\dots<\arg a_d$. Take $d+1$ copies $\Sigma_1,\dots,\Sigma_{d+1}$ of $\widehat\C$; for each $k$ cut $\Sigma_k$ and $\Sigma_{k+1}$ along $\mathcal R_{a_k}$ and reglue crosswise by the identity of the arc. Let $X_A$ be the result and $f_A\colon X_A\to\widehat\C$ the projection.
\end{construction}

\begin{lemma}\label{lem8.3}
$X_A\cong\mathbb S^2$, and $f_A$ is an orientation-preserving branched cover of degree $d+1$ with one branch point of index $2$ over each $a_k$, one of index $d+1$ over $\infty$, and no others. The pullback metric $d_A=d_{f_A}$ is a spherical cone surface with cone angles $4\pi$ over the $a_k$ and $2\pi(d+1)$ over $\infty$, and $f_A$ is $1$-BLD from it.
\end{lemma}
\begin{proof}
The monodromy around $a_k$ is $\sigma_k=(k\ k{+}1)$. The monodromy around $\infty$ is, up to conjugation and inversion, the product $\sigma_1\cdots\sigma_d$, and hence is a $(d+1)$-cycle. The transpositions act transitively, so $X_A$ is connected. Counting the lifted cells of a target triangulation whose vertices contain all branch values gives the Riemann--Hurwitz identity $\chi(X_A)=2(d+1)-d-d=2$. Thus the closed connected oriented surface $X_A$ is a sphere. The metric statements follow from Proposition~\ref{prop4.1} since the gluings are isometries.
\end{proof}

\begin{lemma}\label{lem8.4}
Let $B=B_\rho(y,r)$, $0<r<\pi/4$, with $\partial B\cap(A\cup\{\infty\})=\emptyset$. Every component of $f_A^{-1}(B)$ is an open disk, and its closure is a compact disk with no branch points on its boundary.
\end{lemma}
\begin{proof}
First suppose $\infty\notin B$. Since $\infty\notin\partial B$, the round ball $B$ is a bounded Euclidean disk in the stereographic coordinate. By Sto\"ilow factorization and uniformization~\cite{LP2020,St1928}, there are a homeomorphism $\varphi\colon X_A\to\widehat\C$ and a rational map $P$ with $f_A=P\circ\varphi$. The unique point over $\infty$ has index $d+1$. Choosing the coordinate $\varphi$ to send that point to $\infty$, we see that $P$ has only one pole, at $\infty$, and hence is a polynomial.

Write $B=\{\zeta:|\zeta-b|<s\}$ and let $V$ be a component of $P^{-1}(B)$. It is bounded. If $\gamma\subset V$ is a Jordan curve, the maximum principle applied to $P-b$ on the bounded Jordan domain enclosed by $\gamma$ gives $|P-b|<s$ throughout that domain. The domain and $\gamma$ belong to the same component of $P^{-1}(B)$, so the domain is contained in $V$. The Jordan-curve criterion for planar domains implies that $V$ is simply connected. Moreover, the circle $\partial B$ contains no critical value. Its inverse image is therefore a finite union of smooth circles. Thus $\overline V$ is a compact surface with boundary; simple connectedness makes it a disk. Transport this conclusion by $\varphi^{-1}$.

If $\infty\in B$, then $B\cap A=\emptyset$. Indeed $\rho(\infty,a)>\pi/2$ for $|a|<1$, whereas two points of $B$ have distance less than $2r<\pi/2$. The fiber over $\infty$ is a singleton, so lifting paths to $\infty$ shows that $f_A^{-1}(B)$ is connected. Off that point it is the connected degree-$(d+1)$ cover of a punctured disk. Filling the puncture gives a disk, and the unbranched boundary again gives a compact disk closure. Notice that the argument here uses $2r<\pi/2$, not merely $r<\pi/2$.
\end{proof}

\begin{proposition}\label{prop8.5}
For every $A$ as in Construction~\ref{con8.2}, $(X_A,d_A)$ has contractibility data $(5,1/10)$.
\end{proposition}
\begin{proof}
Let $x\in X_A$, $0<r<1/10$, and $y=f_A(x)$. Choose $r'\in(r,21r/20)$ such that $\partial B(y,r')$ misses every branch value, including $\infty$. Let $W$ be the component of $f_A^{-1}(B(y,r'))$ containing $x$. A path from $x$ of $d_A$-length less than $r$ projects inside $B(y,r')$, so $B_{d_A}(x,r)\subset W$. By Lemma~\ref{lem8.4}, $\overline W$ is a compact disk.

On $\overline W$ pull back the Euclidean metric of the normal-coordinate disk $\exp_y^{-1}(\overline B(y,r'))$, and let $\sigma$ be its intrinsic length metric. This is a compact flat surface with cone angle $2\pi i(z,f_A)\ge2\pi$ at each branch point. The angle can be $2\pi(d+1)$ if the component contains the point over $\infty$. There are no branch points on the boundary, which is locally the convex side of a Euclidean circle. Consequently $\overline W$ is locally $\mathrm{CAT}(0)$, including at its boundary; the cone criterion and the Cartan--Hadamard theorem~\cite[Chapters~II.3--II.4]{BH1999} make the complete simply connected length space $(\overline W,\sigma)$ globally $\mathrm{CAT}(0)$.

The exponential chart on this small round ball is $2$-bi-Lipschitz. For $z\in B_{d_A}(x,r)$, a $d_A$-geodesic from $x$ to $z$ stays in $W$ and has $\sigma$-length less than $2r$. Thus $B_{d_A}(x,r)\subset B_\sigma(x,2r)$. Conversely a path in $\overline W$ of $\sigma$-length less than $2r$ has $d_A$-length less than $4r$, so $B_\sigma(x,2r)\subset B_{d_A}(x,4r)$. Convexity of balls in a $\mathrm{CAT}(0)$ space gives a continuous geodesic contraction to $x$ inside $B_\sigma(x,2r)$. Its restriction to $B_{d_A}(x,r)$ is the required null-homotopy in $B_{d_A}(x,5r)$. All constants are independent of $A$ and $d$.
\end{proof}

\subsection{Proof of Theorem~\ref{thm1.8}}
\begin{proof}
\textup{(i)} We first work in dimension two. Fix $c$ with $|c|=1/2$. Choose $R_0>0$ so small that $B_\rho(c,3R_0)\subset\{0<|z|<1\}$ and that $R_0$ is below the threshold of Theorem~\ref{thm6.1} for the data $n=2$, $L=1$, $\lambda=5$, $r_*=1/10$ and the round target. This threshold is independent of $m$.

Fix $0<R<R_0$ and put $d=m-1$. For sufficiently large $d$, the round ball $B_\rho(c,2R)$ has an $h$-net of at most $d$ points, containing $c$, with
\[
 h=C R/\sqrt d,
\]
where $C$ is an absolute constant. To justify the cardinality, take a maximal $h$-separated set containing $c$. Disjoint balls of radius $h/2$ about its points lie in $B_\rho(c,3R)$, so the round area bounds give at most $C_1(R/h)^2$ points when $h<R$. Choose $C$ with $C_1/C^2<1/2$, and enlarge the lower threshold on $d$ if necessary. Maximality supplies the net property. Perturb the other points by less than $h/10$ so that their arguments, and the argument of $c$, are distinct. Add further points with distinct arguments, if needed, to obtain exactly $d$ points in $B_\rho(c,3R)$. The resulting set $A$ remains a $2h$-net of $B_\rho(c,2R)$ and contains $c$.

Lemma~\ref{lem8.3} and Proposition~\ref{prop8.5} give a degree-$m$ $1$-BLD map from a geodesic sphere with data $(5,1/10)$ and branch values $A\cup\{\infty\}$. Every ball of radius greater than $2h$ contained in $B_\rho(c,R)$ meets $A$, since its center lies in the region covered by the net. This gives the required bound $C_0m^{-1/2}R$, after changing the absolute constant. The finitely many smaller values of $d$ are covered by increasing $C_0$ so that the assertion about contained balls is vacuous for those values, and using any $d$ distinct points including $c$. Because $R$ was arbitrary, the obstruction applies below any positive threshold controlled only by the listed data.

For $n>2$, let $T=\R^{n-2}/\Z^{n-2}$ with its fixed flat metric, and use the $\ell^2$ product length metrics on
\[
 M_m=X_A\times T,\qquad N_n=\widehat\C\times T,
 \qquad F_m=f_A\times\id_T.
\]
For $r<1/10$, a product ball lies in the product of the radius-$r$ balls in its factors. Contracting both factors contracts it inside the radius-$5\sqrt2 r$ product ball. Thus the data $(5\sqrt2,1/10)$ are uniform. The map is open, discrete and orientation-preserving of degree $m$, and its branch-value set is $(A\cup\{\infty\})\times T$. For a rectifiable product path, preservation of the lengths of every subpath in the first factor preserves its metric speed almost everywhere; the product speed formula proves length preservation. If a product path is nonrectifiable, at least one coordinate path is nonrectifiable, and its image still has infinite length. Hence $F_m$ is $1$-BLD in the convention of~\eqref{eq2.3}. A product ball projects onto a round ball of the same radius, so a ball avoiding the branch-value set projects onto one avoiding $A\cup\{\infty\}$. The two-dimensional net estimate applies. Decrease $R_0$ for each $n$ to be below the threshold of Theorem~\ref{thm6.1} for these product data.

\textup{(ii)} For a degree-$m$ branched cover of $\mathbb S^2$ by $\mathbb S^2$, Sto\"ilow's theorem and the Riemann--Hurwitz formula give
\[
 \#f(B_f)\le\#B_f\le\sum_{x\in B_f}(i(x,f)-1)=2m-2.
\]
At every fixed sufficiently small round scale $R$, at least $c_1(R/\varrho)^2$ disjoint balls of radius $\varrho$ can be packed in $B(y_0,R)$. With $\varrho=cm^{-1/2}R$ and $c>0$ small enough, their number exceeds $2m-2$, so one ball contains no branch value. This proves the claimed uniform lower bound. The examples in \textup{(i)} prove its sharp order for the sphere-to-sphere class.
\end{proof}

\begin{remark}\label{rem8.6}
For sphere-to-sphere covers the exponential dependence in the particular constant of Theorem~\ref{thm6.1} is therefore not sharp; the optimal order in that class is $m^{-1/2}$. This does not identify the optimal constant for general closed surface sources or for $n\ge3$. Lemma~\ref{lem6.5} distinguishes a lower bound for the diameter of some fiber from a lower bound for the separation of all its points. The recursion supplies the latter indirectly by splitting inverse-image components, and can lose quantitative information at every step.
\end{remark}

\section{Open questions}\label{sec9}

\subsection{Optimal dependence of porosity on the degree}
With fixed geometric data and an $m$-independent scale threshold, Theorem~\ref{thm6.1} gives $\delta_m=\tfrac12(a/(m+1))^{m-1}$, where $a=\varepsilon_0/[8L(2\sqrt n+L)]$. Theorem~\ref{thm1.8} gives the sharp planar order $m^{-1/2}$; its product examples exclude slower decay for the general manifold class in higher dimensions.

\begin{question}\label{que9.1}
In the class of Theorem~\ref{thm1.5} with fixed $n\ge3$, $N$, $L$, $\lambda$, and $r_*$, can the porosity bound at an $m$-independent positive scale threshold be improved to $c\,m^{-\alpha}$, with $c,\alpha>0$ independent of $m$?
\end{question}

\subsection{A geometric criterion for the embedded branch locus}
Fix $n\ge3$ and $m\ge2$, let $h\colon\Sn\to\Sn$ be an orientation-preserving homeomorphism, and put $A=h(E)$.

\begin{lemma}\label{lem9.2}
If $h'\colon\Sn\to\Sn$ is another orientation-preserving homeomorphism with $h'(E)=A$, there is a source homeomorphism $H$ such that $h\circ q_m=h'\circ q_m\circ H$ and
\[
 d_{h\circ q_m}(x,x')=d_{h'\circ q_m}(H(x),H(x')).
\]
Thus admissibility depends only on the embedded set $A$ and the degree $m$.
\end{lemma}
\begin{proof}
Put $\Phi=h'^{-1}\circ h$, so $\Phi(E)=E$. The complement $\Sn\setminus E$ has fundamental group $\Z$, and its winding cover corresponds to $m\Z$. Since $\Phi_*(m\Z)=m\Z$, the covering-space lifting criterion~\cite[Proposition~1.33]{Ha2002} gives a covering isomorphism $H$ with $\Phi\circ q_m=q_m\circ H$ off $E$. Each fiber over $E$ is a singleton, so compactness extends $H$ and its inverse continuously over $E$. The resulting homeomorphism satisfies the asserted factorization, and bijecting paths by $H$ in~\eqref{eq1.1} proves the metric identity.
\end{proof}

Recall that $A$ has \emph{bounded turning} if every $a,b\in A$ lie in a subcontinuum $K\subset A$ with $\diam_\rho K\le C\rho(a,b)$ for one constant $C$~\cite{H2001}. Corollary~\ref{cor6.3} requires uniform porosity for admissibility, so we impose that condition explicitly, without assuming that bounded turning alone implies ambient porosity.

\begin{question}\label{que9.3}
For $n=3$ and a uniformly porous circle $A=h(E)$, is $h\circ q_m$ admissible if and only if $A$ has bounded turning?
\end{question}

The porous cusp example fails bounded turning, whereas the quasisymmetric examples satisfy it. Neither implication for an arbitrary $h$ is proved here: the essential-loop argument requires control of the topology of inverse-image neighborhoods.

\subsection{Universal preservation under target changes}
\begin{remark}\label{rem9.4}
If $f$ is admissible and target and source homeomorphisms $\Phi,H$ satisfy $\Phi\circ f=f\circ H$, then $d_{\Phi\circ f}(x,x')=d_f(H(x),H(x'))$, so $\Phi\circ f$ is admissible. The lifting argument in Lemma~\ref{lem9.2} applies to every $\Phi$ preserving $E$. Such maps, including homeomorphisms supported in a ball disjoint from $E$, need not be quasisymmetric.
\end{remark}

\begin{question}\label{que9.5}
Fix $n\ge3$. If an orientation-preserving homeomorphism $\Phi\colon\Sn\to\Sn$ makes $\Phi\circ f$ admissible for every admissible orientation-preserving branched cover $f\colon\Sn\to\Sn$ of degree at least two, must $\Phi$ be quasisymmetric?
\end{question}

Theorem~\ref{thm7.5} proves the converse implication. Degree-one covers are admissible after every target homeomorphism, so only branching can test this property.

\section*{Acknowledgments}
This work was supported by the Visiting Scholar Program of Chern Institute of Mathematics at Nankai University. The author would
like to express his hearty thanks to Chern Institute of Mathematics provided very comfortable research environments to him worked as visiting scholar.
This work was also supported by the Guangdong Basic and Applied Basic Research Foundation (Nos.\ 2022A1515110967 and 2023A1515011809). The author used GPT 5.6 sol for editorial, organizational and typesetting assistance. All mathematical statements, derivations and references were checked and are the responsibility of the author.


\begin{thebibliography}{99}

\bibitem{A1964}
L. V. Ahlfors,
Extension of quasiconformal mappings from two to three dimensions,
\emph{Proc. Nat. Acad. Sci. U.S.A.} \textbf{51} (1964), no.~5, 768--771.

\bibitem{AP2000}
S. Alpern and V. S. Prasad,
\emph{Typical Dynamics of Volume Preserving Homeomorphisms},
Cambridge Tracts in Mathematics, vol.~139, Cambridge University Press, Cambridge, 2000.

\bibitem{BMW2025}
G. Basso, D. Marti, and S. Wenger,
Geometric and analytic structures on metric spaces homeomorphic to a manifold,
\emph{Duke Math. J.} \textbf{174} (2025), no.~9, 1723--1773.

\bibitem{BB2018}
A. Bj\"orn and J. Bj\"orn,
Local and semilocal Poincar\'e inequalities on metric spaces,
\emph{J. Math. Pures Appl.} \textbf{119} (2018), 158--192.

\bibitem{BH2004}
M. Bonk and J. Heinonen,
Smooth quasiregular mappings with branching,
\emph{Publ. Math. Inst. Hautes \'Etudes Sci.} \textbf{100} (2004), 153--170.

\bibitem{BH1999}
M. R. Bridson and A. Haefliger,
\emph{Metric Spaces of Non-positive Curvature},
Grundlehren der mathematischen Wissenschaften, vol.~319, Springer-Verlag, Berlin, 1999.

\bibitem{BBI2001}
D. Burago, Y. Burago, and S. Ivanov,
\emph{A Course in Metric Geometry},
Graduate Studies in Mathematics, vol.~33, American Mathematical Society, Providence, RI, 2001.

\bibitem{C1964}
A. V. \v Cernavski\u\i,
Finite-to-one open mappings of manifolds,
\emph{Mat. Sb. (N.S.)} \textbf{65(107)} (1964), 357--369.

\bibitem{CH1960}
P. T. Church and E. Hemmingsen,
Light open maps on $n$-manifolds,
\emph{Duke Math. J.} \textbf{27} (1960), 527--536.

\bibitem{DPT2026}
G. C. David, P. Pankka, and J. T. Tyson,
The Heinonen--Semmes problems after thirty years,
\emph{Conform. Geom. Dyn.} \textbf{30} (2026), 86--134.

\bibitem{Fa1980}
A. Fathi,
Structure of the group of homeomorphisms preserving a good measure on a compact manifold,
\emph{Ann. Sci. \'Ecole Norm. Sup. (4)} \textbf{13} (1980), no.~1, 45--93.

\bibitem{F1950}
E. E. Floyd,
Some characterizations of interior maps,
\emph{Ann. of Math. (2)} \textbf{51} (1950), 571--575.

\bibitem{GV1973}
F. W. Gehring and J. V\"ais\"al\"a,
Hausdorff dimension and quasiconformal mappings,
\emph{J. London Math. Soc. (2)} \textbf{6} (1973), no.~3, 504--512.

\bibitem{GW2016a}
C.-Y. Guo and M. Williams,
The branch set of a quasiregular mapping between metric manifolds,
\emph{C. R. Math. Acad. Sci. Paris} \textbf{354} (2016), no.~2, 155--159.

\bibitem{GW2016b}
C.-Y. Guo and M. Williams,
Geometric function theory: the art of pullback factorization,
\emph{arXiv:1611.02478v1} (2016).

\bibitem{GW2017}
C.-Y. Guo and M. Williams,
Porosity of the branch set of discrete open mappings with controlled linear dilatation,
\emph{arXiv:1509.02638v2} (2017); first version posted in 2015.

\bibitem{HM2015}
P. Haj\l asz and S. Malekzadeh,
A new characterization of the mappings of bounded length distortion,
\emph{Int. Math. Res. Not. IMRN} \textbf{2015} (2015), no.~24, 13238--13244.

\bibitem{HMZ2018}
P. Haj\l asz, S. Malekzadeh, and S. Zimmerman,
Weak BLD mappings and Hausdorff measure,
\emph{Nonlinear Anal.} \textbf{177} (2018), 524--531.

\bibitem{Ha2002}
A. Hatcher,
\emph{Algebraic Topology},
Cambridge University Press, Cambridge, 2002.

\bibitem{H2001}
J. Heinonen,
\emph{Lectures on Analysis on Metric Spaces},
Universitext, Springer-Verlag, New York, 2001.

\bibitem{HK1998}
J. Heinonen and P. Koskela,
Quasiconformal maps in metric spaces with controlled geometry,
\emph{Acta Math.} \textbf{181} (1998), no.~1, 1--61.

\bibitem{HKST2015}
J. Heinonen, P. Koskela, N. Shanmugalingam, and J. T. Tyson,
\emph{Sobolev Spaces on Metric Measure Spaces: An Approach Based on Upper Gradients},
New Mathematical Monographs, vol.~27, Cambridge University Press, Cambridge, 2015.

\bibitem{HR1998}
J. Heinonen and S. Rickman,
Quasiregular maps $S^3\to S^3$ with wild branch sets,
\emph{Topology} \textbf{37} (1998), no.~1, 1--24.

\bibitem{HR2002}
J. Heinonen and S. Rickman,
Geometric branched covers between generalized manifolds,
\emph{Duke Math. J.} \textbf{113} (2002), no.~3, 465--529.

\bibitem{HS1997}
J. Heinonen and S. Semmes,
Thirty-three yes or no questions about mappings, measures, and metrics,
\emph{Conform. Geom. Dyn.} \textbf{1} (1997), 1--12.

\bibitem{KTW2005}
R. Kaufman, J. T. Tyson, and J.-M. Wu,
Smooth quasiregular maps with branching in $\R^n$,
\emph{Publ. Math. Inst. Hautes \'Etudes Sci.} \textbf{101} (2005), 209--241.

\bibitem{Lee2018}
J. M. Lee,
\emph{Introduction to Riemannian Manifolds}, second ed.,
Graduate Texts in Mathematics, vol.~176, Springer, Cham, 2018.

\bibitem{LR2014}
F. Le Roux,
On closed subgroups of the group of homeomorphisms of a manifold,
\emph{J. \'Ec. polytech. Math.} \textbf{1} (2014), 147--159.

\bibitem{L2016}
R. Luisto,
Note on local-to-global properties of BLD-mappings,
\emph{Proc. Amer. Math. Soc.} \textbf{144} (2016), no.~2, 599--607.

\bibitem{L2017}
R. Luisto,
A characterization of BLD-mappings between metric spaces,
\emph{J. Geom. Anal.} \textbf{27} (2017), no.~3, 2081--2097.

\bibitem{L2019}
R. Luisto,
A Newman property for BLD-mappings,
\emph{Conform. Geom. Dyn.} \textbf{23} (2019), 135--146.

\bibitem{LP2020}
R. Luisto and P. Pankka,
Sto\"ilow's theorem revisited,
\emph{Expo. Math.} \textbf{38} (2020), no.~3, 303--318.

\bibitem{Luu1998}
J. Luukkainen,
Assouad dimension: antifractal metrization, porous sets, and homogeneous measures,
\emph{J. Korean Math. Soc.} \textbf{35} (1998), no.~1, 23--76.

\bibitem{MR1973}
O. Martio and S. Rickman,
Measure properties of the branch set and its image of quasiregular mappings,
\emph{Ann. Acad. Sci. Fenn. Ser. A I} \textbf{541} (1973), 16 pp.

\bibitem{MV1988}
O. Martio and J. V\"ais\"al\"a,
Elliptic equations and maps of bounded length distortion,
\emph{Math. Ann.} \textbf{282} (1988), no.~3, 423--443.

\bibitem{M1995}
P. Mattila,
\emph{Geometry of Sets and Measures in Euclidean Spaces: Fractals and Rectifiability},
Cambridge Studies in Advanced Mathematics, vol.~44, Cambridge University Press, Cambridge, 1995.

\bibitem{MR1983}
L. F. McAuley and E. E. Robinson,
On Newman's theorem for finite-to-one open mappings on manifolds,
\emph{Proc. Amer. Math. Soc.} \textbf{87} (1983), no.~3, 561--566.

\bibitem{MR1984}
L. F. McAuley and E. E. Robinson,
Discrete open and closed mappings on generalized continua and Newman's property,
\emph{Canad. J. Math.} \textbf{36} (1984), no.~6, 1081--1112.

\bibitem{N1931}
M. H. A. Newman,
A theorem on periodic transformations of spaces,
\emph{Q. J. Math. Oxf. Ser.} \textbf{2} (1931), 1--8.

\bibitem{OU1941}
J. C. Oxtoby and S. M. Ulam,
Measure-preserving homeomorphisms and metrical transitivity,
\emph{Ann. of Math. (2)} \textbf{42} (1941), 874--920.

\bibitem{R1973}
S. Rickman,
Path lifting for discrete open mappings,
\emph{Duke Math. J.} \textbf{40} (1973), 187--191.

\bibitem{R1993}
S. Rickman,
\emph{Quasiregular Mappings},
Ergebnisse der Mathematik und ihrer Grenzgebiete (3), vol.~26, Springer-Verlag, Berlin, 1993.

\bibitem{Sa1975}
J. Sarvas,
The Hausdorff dimension of the branch set of a quasiregular mapping,
\emph{Ann. Acad. Sci. Fenn. Ser. A I Math.} \textbf{1} (1975), no.~2, 297--307.

\bibitem{S1996}
S. Semmes,
Finding curves on general spaces through quantitative topology, with applications to Sobolev and Poincar\'e inequalities,
\emph{Selecta Math. (N.S.)} \textbf{2} (1996), no.~2, 155--295.

\bibitem{St1928}
S. Sto\"ilow,
Sur les transformations continues et la topologie des fonctions analytiques,
\emph{Ann. Sci. \'Ecole Norm. Sup. (3)} \textbf{45} (1928), 347--382.

\bibitem{TV1982}
P. Tukia and J. V\"ais\"al\"a,
Quasiconformal extension from dimension $n$ to $n+1$,
\emph{Ann. of Math. (2)} \textbf{115} (1982), no.~2, 331--348.

\bibitem{V1966}
J. V\"ais\"al\"a,
Discrete open mappings on manifolds,
\emph{Ann. Acad. Sci. Fenn. Ser. A I} \textbf{392} (1966), 10 pp.

\bibitem{V1971}
J. V\"ais\"al\"a,
\emph{Lectures on $n$-Dimensional Quasiconformal Mappings},
Lecture Notes in Mathematics, vol.~229, Springer-Verlag, Berlin, 1971.

\end{thebibliography}
\end{document}